\documentclass[10pt, english]{amsart}

\usepackage{amsmath,amssymb,enumerate}

\usepackage[T1]{fontenc}
\usepackage[all]{xy}

\usepackage{babel}
\usepackage{amstext}
\usepackage{amsmath}
\usepackage{amsfonts}
\usepackage{latexsym}
\usepackage{ifthen}

\usepackage{xypic}
\xyoption{all}
\newtheorem{lemma1}{}[section]

\newenvironment{lemma}{\begin{lemma1}{\bf Lemma.}}{\end{lemma1}}
\newenvironment{example}{\begin{lemma1}{\bf Example.}\rm}{\end{lemma1}}

\newenvironment{theorem}{\begin{lemma1}{\bf Theorem.}}{\end{lemma1}}

\newenvironment{proposition}{\begin{lemma1}{\bf Proposition.}}{\end{lemma1}}
\newenvironment{corollary}{\begin{lemma1}{\bf Corollary.}}{\end{lemma1}}
\newenvironment{remark}{\begin{lemma1}{\bf Remark.}\rm}{\end{lemma1}}

\newenvironment{definition}{\begin{lemma1}{\bf Definition.}}{\end{lemma1}}
\newenvironment{notation}{\begin{lemma1}{\bf Notation.}}{\end{lemma1}}

\newenvironment{setup}{\begin{lemma1}{\bf Setup.}}{\end{lemma1}}
\newenvironment{construction}{\begin{lemma1}{\bf Construction.}\rm}{\end{lemma1}}
\newenvironment{conjecture}{\begin {lemma1}{\bf Conjecture.}}{\end{lemma1}}

\newenvironment{assumption}{\begin{lemma1}{\bf Assumption.}}{\end{lemma1}}

\newenvironment{the local obstruction - setup}{\begin{lemma1}{\bf The local obstruction - setup.}}{\end{lemma1}}

\newenvironment{remark*}{{\bf Remark.}}{}
\newenvironment{remarks*}{{\bf Remarks.}}{}
\newenvironment{example*}{{\bf Example.}}{}
\newenvironment{assumption*}{{\bf Assumption.}}{}

\newcommand{\Q}{\ensuremath{\mathbb{Q}}}
\newcommand{\Z}{\ensuremath{\mathbb{Z}}}
\newcommand{\C}{\ensuremath{\mathbb{C}}}
\newcommand{\N}{\ensuremath{\mathbb{N}}}
\newcommand{\PP}{\ensuremath{\mathbb{P}}}

\usepackage{eurosym}

\newcommand{\holom}[3]{\ensuremath{#1:#2  \rightarrow #3}}
\newcommand{\fibre}[2]{\ensuremath{#1^{-1} (#2)}}

\makeatletter
\ifnum\@ptsize=0  \fi \ifnum\@ptsize=2  \fi 

\newcommand\sE{{\mathcal E}}

\newcommand\sF{{\mathcal F}}
\newcommand\sG{{\mathcal G}}
\newcommand\sH{{\mathcal H}}
\newcommand\sI{{\mathcal I}}
\newcommand\sT{{\mathcal T}}

\newcommand\sO{{\mathcal O}}

\DeclareMathOperator*{\sing}{sing}

\DeclareMathOperator*{\nons}{nons}

\usepackage{xcolor}

\title{A Nonvanishing Conjecture for Cotangent Bundles} 
\date{\today}

\author{Andreas H\"oring}
\author{Thomas Peternell}

\address{Andreas H\"oring, Universit\'e C\^ote d'Azur, CNRS, LJAD, France, Institut universitaire de France}
\email{Andreas.Hoering@univ-cotedazur.fr}

\address{Thomas Peternell, Mathematisches Institut, Universit\"at Bayreuth, 95440 Bayreuth, 
Germany}
\email{thomas.peternell@uni-bayreuth.de}

\subjclass[2010]{14J32, 37F75, 14E30}
\keywords{MMP, minimal models, positivity of vector bundles, nonvanishing conjecture, symmetric differentials}

\begin{document}

\begin{abstract}  
In this paper we study the positivity of the cotangent bundle of projective manifolds.
We conjecture that the cotangent bundle is pseudoeffective if and only the manifold
has non-zero symmetric differentials. We confirm this conjecture for most projective surfaces that are not of general type.
\end{abstract}

\newpage

\maketitle

\section{Introduction} 

\subsection{Main result}
A central part in the minimal model program in algebraic geometry is the so-called nonvanishing conjecture:  given a projective manifold or,
more generally, a variety with klt singularities, $X$, whose canonical class $K_X$ is pseudoeffective, one has 
$$
H^0(X,\sO_X(mK_X)) \ne 0
$$
for some positive integer $m$. This conjecture has been proven some time ago in dimension at most three, but is wide open in higher dimensions. 

In analogy to the nonvanishing conjecture, one might ask for 

\begin{conjecture} \label{conj:nv} 
Let $X$ be a normal projective variety with klt singularities. 
Let $1 \leq q \leq \dim X$. Then $\Omega^{[q]}_X$, the sheaf of reflexive holomorphic differentials in degree $q$, is pseudoeffective,
(see Definition \ref{def:reflexive}),  if and only 
for some positive integer $m$ one has  
$$
H^0(X,S^{[m]}\Omega^{[q]}_{X}) \ne 0.
$$ 
\end{conjecture} 

In the case $q = \dim X$, this is of course the nonvanishing conjecture stated above.
The only general result confirming Conjecture \ref{conj:nv} is given in \cite[Thm.1.6]{HP19}: Suppose $X$ is klt and smooth in codimension two with $K_X \equiv 0$.
If $\Omega^{[1|}_X $ is pseudoeffective, there is a quasi-\'etale cover $\widetilde  X \to X$ such that 
$q(\widetilde X) > 0$. In particular one has $ H^0(X,S^{[m]}\Omega^{[1]}_{X}) \ne 0$ for some positive integer $m$ (\cite[Prop.2.2]{Ane18}, see also Lemma \ref{propcover2}). 
While the pseudoeffectivity of $K_X$ is equivalent to the non-uniruledness of the manifold, 
we do not know many examples where $\Omega_X^q$ is pseudoeffective, but not big. 
We expect that this property is actually quite restrictive, our Theorem \ref{theoremmain} confirms this intuition in the first non-trivial case.

In this paper we are mainly interested in the case $q = 1$. Already for smooth surfaces, Conjecture \ref{conj:nv} is delicate. In this case we can assume without loss of generality
that $X$ is minimal, see Proposition \ref{prop:bir}.
By surface classification, see Corollary \ref{corollarysurfaces},
the problem starts with $\kappa (X) = 1$. We basically settle this case:

\begin{theorem} \label{theoremmain} 
Let $f: X \to B$ be a (minimal) smooth elliptic surface with $\kappa (X) = 1$ such that $\Omega^1_X$ is pseudoeffective. Suppose 
one of the following.
\begin{enumerate}
\item $f$ is not isotrivial
\item $f$ is isotrivial and the general fiber does not have complex multiplication
\item The tautological class on $\mathbb P(\Omega_X^1) $ is nef in codimension one.
\end{enumerate}
Then $\widetilde q(X) > 0$ (see Definition \ref{definitiontildeq}), 
so there is a positive integer $m$ such that 
$$
H^0(X,S^m\Omega_X^1)) \ne 0.
$$ 
\end{theorem} 

Note that each of the cases requires a different proof, Theorem \ref{theoremmain} is obtained as the union of Corollary \ref{cor:noniso}, Corollary \ref{cor:isotrivialstandard} and
Corollary \ref{cor:Zar}.

In general, the pseudoeffectivity of  $\Omega^{[1]}_X$ does not imply $\tilde q(X) > 0$: 
a smooth complete intersection surface $X \subset \PP^N$ is simply connected, so $\tilde q(X)=0$. However, if $N \geq 4$ and the multidegrees are sufficiently high, the
cotangent bundle $\Omega^1_X$ is ample \cite{Bro14}.

For surfaces of general type, Conjecture \ref{conj:nv} is open. 
If $c_1^2(X) > c_2(X) $, then by Bogomolov's vanishing theorem
$$
h^0(X, S^m\Omega_X^1)) \sim \bigl( \frac{c_1^2(X)-c_2(X)}{6} \bigr ) m^3, 
$$
but already the boundary case $c_1^2(X) = c_2(X)$ is unclear. 

In higher dimension, things get worse due to the singularities of minimal models. For example, we know Conjecture \ref{conj:nv} for terminal threefolds
with numerically trivial canonical class, but we cannot deduce easily Conjecture \ref{conj:nv} for smooth threefolds $X$ with $\kappa (X) = 0$, although
$X$ has a terminal minimal model as above. 

We would finally like to point out the connection to a question posed by H.Esnault, see \cite{BKT13}:  
let $X$ be a projective (or compact K\"ahler) manifold whose fundamental group
$\pi_1(X)$ is infinite. Does there exist a positive integer $m$ such that 
$$ 
H^0(X,S^m\Omega^1_X) \ne 0?
$$ 
An intermediate step might be to prove that $\Omega^1_X$ is pseudoeffective. 
Brunebarbe-Klingler-Totaro confirm Esnault's conjecture
if  there is a representation 
$$\pi_1(X) \to \mbox{GL}(N,\mathbb C)$$
with infinite image. The key point of their proof is to show that in many cases, the cotangent bundle $\Omega^1_X$ is even big. 

Theorem \ref{theoremmain} fits in the framework of Esnault's conjecture: it is well-known to experts that if the cotangent bundle of an elliptic surface is not pseudoeffective, then $\pi_1(X)$ is finite (see Appendix \ref{appendixfinite} for a proof). In view of Theorem \ref{theoremmain}, the two properties should actually be equivalent (which is actually the case up to the exceptional isotrivial case of Theorem \ref{theoremmain}) and imply the existence of symmetric differentials.

\subsection{Strategy of the proof}
Let $X$ be a smooth projective surface such that $K_X$ is nef and $c_1(K_X)^2=0$. Then $K_X$ is semiample, so we have the Iitaka fibration
$$
\holom{f}{X}{B}
$$
such that the general fibre $F$ is elliptic and $m K_X \simeq f^* A$ with $A$ an ample divisor. Assume now that $\Omega_X := \Omega^1_X$ is pseudoeffective, then one expects that there exists a pseudoeffective subsheaf of $\Omega_X$ that is induced by a pull-back
from the base $B$. Let $f^* \Omega_B \rightarrow \Omega_X$ be the cotangent map, and denote by 
$$
f^* \Omega_B(D) \subset \Omega_X
$$
the saturation. If $f^* \Omega_B(D)$ is pseudoeffective, Proposition \ref{propA} shows that $\tilde q(X)>0$. Thus the main issue in Theorem \ref{theoremmain} is to show that $f^* \Omega_B(D)$ is pseudoeffective. A natural approach is to show that the
sheaf $\Omega_X \rightarrow \omega_{X/B}(-D)$ is not  pseudoeffective if $f$ is not almost smooth\footnote{See the introduction of Section \ref{sectionelliptic} for the almost smooth case.}.
However, by a theorem of Brunella \cite{Bru06}, the line bundle $\omega_{X/B}(-D)$ is always pseudoeffective! This leads us to considering the 
more refined quotient sheaf
$$
\Omega_X \rightarrow \sI_Z \otimes \omega_{X/B}(-D) \rightarrow 0,
$$
where $Z$ has support in the singular points of the reduction of the fibres. The basic idea of the proof of Theorem \ref{theoremmain} 
is to show that the torsion-free sheaf $\sI_Z \otimes \omega_{X/B}(-D)$ is not strongly pseudoeffective (see Definition \ref{definitionpseff}), although 
its bidual is a pseudoeffective line bundle. Thanks to a result of Demailly-Peternell-Schneider \cite{DPS94} this idea leads immediately
to the result in the non-isotrivial case, see Section \ref{sectionelliptic}.  
For an isotrivial fibration this approach only yields the weaker statement appearing
as part c) of the main theorem, see Subsection \ref{subsectionisotrivialzariski}.
On the other hand we know that $X$ is birational to a quotient $(C \times E)/G$, 
so we aim to compute explicitly the spaces of global sections
$$
H^0(X, S^i \Omega_X \otimes \sO_X(jA)),
$$
following a strategy introduced by Sakai \cite{Sak79}. Apart from the technical setup, the main difficulty is to understand the local obstruction near the fixed points of the group action. For $A_1$-singularities this information is provided by \cite[Prop.3.2.]{BTV19}.
We expect that a similar description of the local obstruction for klt singularities 
would allow to handle the case when the elliptic curve has complex multiplication.

\subsection{Structure of the paper}

In Section 3 we introduce a positivity notion (``strongly pseudoeffective'') that is adapted for this type of torsion-free sheaf,
and present material on pseudoeffective torsion free sheaves which will be used in later sections. 

Section 4 is concerned with some general results on varieties with pseudoeffective cotangent sheaves. In particular, generalised Kodaira dimensions are introduced 
and a relation to the MRC fibration is studied. The last two sections are devoted to the proof of Theorem \ref{theoremmain}. Section 5 gives the general setup 
and settles the case that the elliptic fibration is not isotrivial. 
The surprisingly difficult isotrivial case finally is studied in Section \ref{sectionisotrivial}.

{\it Acknowledgements.} 
The first-named author thanks the Institut Universitaire de France and the A.N.R. project Foliage (ANR-16-CE40-0008) for providing excellent working conditions. The second-named author was supported by
a DFG grant "Zur Positivit\"at in der komplexen Geometrie". We thank C. Gachet for pointing out the Example \ref{examplenotstrongly}.

\section{Basic notations}

We work over the complex numbers, for general definitions we refer to \cite{Har77}. 
We use the terminology of \cite{Deb01} and \cite{KM98}  for birational geometry and notions from the minimal model program and \cite{Laz04a} for notions of positivity.
Manifolds and varieties will always be supposed to be irreducible and reduced.

\begin{notation} {\rm Let $X$ be a normal complex variety. As usual, $\Omega^1_X$ denotes the sheaf of K\"ahler differentials, and we set 
$$ 
\Omega^{[q]}_X := (\bigwedge^q\Omega^1_X)^{**}.
$$
If $X$ is klt and $\mu:\hat X \to X$ a is desingularization, then by \cite[Thm.1.4]{GKKP11},
$$ \Omega^{[q]}_X = \mu_*(\Omega^q_{\hat X}).$$
If $q = 1$ and $X$ smooth, we simply set $\Omega_X := \Omega^1_X$. 
Finally, for any normal variety $X$,  we denote by $T_X:= (\Omega_X^1)^{*}$ its tangent sheaf. 
}
\end{notation}

A finite surjective map $\holom{\gamma}{X'}{X}$ between normal varieties is {\it quasi-\'etale} if its ramification divisor
is empty (or equivalently, by purity of the branch locus, $\gamma$ is \'etale over the smooth locus of $X$).

\begin{definition} \label{definitiontildeq} Let $X$ be a normal projective variety with klt singularities. Then, as usual,
$$ q(X) = h^1(X,\sO_X) = h^0(X,\Omega^{[1]}_X) $$
is the irregularity of $X$. 
Further, we denote  by $\tilde q(X) $ the maximal irregularity $q(\tilde X)$, where
$\tilde X \to X$ is any quasi-\'etale cover. 
\end{definition}

While the irregularity $q(X)$ is a 
birational invariant of projective varieties with klt singularities, this is not the case for $\tilde q(X)$:

\begin{example} \label{exampleproducttype}
Let $\holom{\tau}{E_1}{\PP^1}$ be a (hyper)elliptic curve, and denote by $i_{E_1}$ the involution induced by the double cover.
Let $E_2$ be an elliptic curve, and denote by $i_{E_2}$ the involution defined by $z \mapsto -z$.
The surface  
$$
X' :=(E_1 \times E_2)/\langle i_{E_1} \times i_{E_2} \rangle
$$
is  normal and has
$A_1$-singularities in the branch points of the quasi-\'etale map $E_1 \times E_2 \rightarrow X'$.
The projection on the first factor induces an isotrivial elliptic fibration 
$$
f': X' \rightarrow \PP^1=E_1/\langle i_{E_1} \rangle
$$
that has exactly
$2g(E_1)+2$ singular fibres, all of them are multiple fibres of multiplicity $2$
such that the reduction is isomorphic to $\PP^1=E_2/\langle i_{E_2} \rangle$. 
By construction we have $\tilde q(X') \geq g(E_1) + g(E_2)>0$.

Denote by $\holom{\mu}{X}{X'}$ the minimal resolution, then the induced elliptic fibration $\holom{f}{X}{\PP^1}$ is relatively minimal,
isotrivial and has exactly $2g(E_1)+2$ singular fibres, all of them of type $I_0^*$ (in Kodaira's terminology, see \cite[V, Table 3]{BHPV04}).

In the classical case where $E_1$ is an elliptic curve, the surface $X$ is a K3 surface of Kummer type. In particular we have 
$\tilde q(X)=0$.
\end{example}

\section{Pseudoeffective sheaves}

\begin{notation} {\rm Let $\sG$ be a coherent sheaf on a variety $X$, and let $\sT \subset \sG$ its torsion subsheaf. Then we denote by
$\sG/\mbox{Tor}$ the quotient $\sG/\sT$. 
Furthermore, we set  $$S^{[m]}(\sG) := (S^m \sG)^{**}.$$
}
\end{notation} 

\subsection{Projectivization of sheaves} 

\begin{definition} \label{definitionproj}
Let $\sF$ be a coherent sheaf on a variety $X$. Then we denote by $\holom{\pi}{\PP(\sF)}{X}$ the projectivisation of $\sF$ in the sense of \cite[II, \S 2, Sect.2]{AT82}. 

We denote by $\zeta_{\PP(\sF)}$ (or $\zeta$ when no confusion is possible) the Cartier divisor class associated to the tautological
line bundle $\sO_{\PP(\sF)}(1)$.
\end{definition}

\begin{remark*}
The reduction of any fibre of $\pi$ is a projective space, and $\pi$ is locally trivial if and only if $\sF$ is locally free \cite[p.27]{AT82}.
\end{remark*}

For locally free sheaves, the following definition of pseudoeffectivity is  now in common.

\begin{definition} \label{definitionpseffvb}
Let $X$ be a  projective variety, 
and let $\sF$ be a locally free sheaf on $X$.  Denote by $\holom{\pi}{\PP(\sF)}{X}$ the projectivisation, and by $\zeta$ the tautological class on $\PP(\sF)$. 
We say that $\sF$ is pseudoeffective if $\zeta$ is a pseudoeffective Cartier divisor class.
\end{definition}

\begin{remark*}
By \cite[Lemma 2.7]{Dru18}, the locally free sheaf
$\sF$ is pseudoeffective if and only if 
for some ample Cartier divisor $H$ on $X$ and for all $c>0$ there exist numbers $j \in \N$ and $i \in \N$ such that $i>cj$ and
$$
H^0(X, S^i \sF \otimes \sO_X(jH)) \neq 0.
$$
\end{remark*}

We will use the following lemma, which will be generalised below. 

\begin{lemma} \label{lemma:func} Let $f: X \to Y$ be a surjective morphism of projective varieties and $\sF$ a locally free sheaf on $Y$. Then 
$\sF$ is pseudoeffective if and only if $f^*(\sF)$ is pseudoeffective. 
\end{lemma} 

\begin{remark*}
Note that the statement applies in particular to the normalisation, so for locally free sheaves pseudoeffectivity can be verified 
on the normalisation.
\end{remark*}

\begin{proof} 
Recall the  pull-back formula for the tautological classes
$$ \zeta_{\mathbb P(f^*(\sF))} = p^*(\zeta_{\mathbb P(\sF)}),$$
where $p: \mathbb P(f^*(\sF)) = \mathbb P(\sF) \times_Y X \to \mathbb P(\sF)$ is the canonical projection.
Thus we are reduced to the case where $\sF$ has rank one, which is immediate by \cite[Thm.2.2.26, Prop.2.2.43]{Laz04a}.
\end{proof}

\subsection{Strongly pseudoeffective torsion-free sheaves} 

For the purpose of this paper it is not sufficient to discuss the positivity of locally free sheaves, in fact we will need the more subtle
positivity properties of torsion-free sheaves. It will suffice to consider normal varieties. 

We first recall the definition of pseudoeffectivity for reflexive sheaves from \cite{Dru18} and \cite[Defn.2.1]{HP19}.

\begin{definition} \label{def:reflexive}  
Let $X$ be a normal projective variety, 
and let $\sF$ be a reflexive sheaf on $X$.  
Then $\sF$ is pseudoeffective
if for some ample Cartier divisor $H$ on $X$ and for all $c>0$ there exist numbers $j \in \N$ and $i \in \N$ such that $i>cj$ and
$$
H^0(X, S^{[i]} \sF \otimes \sO_X(jH)) \neq 0.
$$ 
\end{definition} 

\begin{remark*}
An equivalent definition using an adapted resolution of singularities of $\PP(\sF)$ is given in \cite{HP19}. 
\end{remark*}

\begin{example} \label{examplesubsheaf}
Our definition of pseudoeffectivity is less restrictive than \cite[Defn.7.1]{BDPP13}: if $\sG \subset \sF$ is a pseudoeffective reflexive subsheaf, then $\sF$ is pseudoeffective. In particular if $\sF = L \oplus H$ where $L$ is pseudoeffective and $H$ an antiample reflexive sheaf, then 
$\sF$ is pseudoeffective in the sense of Definition \ref{def:reflexive}, but not in the sense of \cite[Defn.7.1]{BDPP13}.
\end{example}

Definition \ref{def:reflexive} makes also sense for torsion-free sheaves, but would not be very useful: by definition a torsion-free sheaf would
be pseudoeffective if and only if its bidual is pseudoeffective. The following definition takes this difference into account:

\begin{definition} \label{definitionpseff}
Let $X$ be a normal projective variety, 
and let $\sF$ be a torsion free sheaf on $X$.  
We say that $\sF$ is  strongly pseudoeffective
if for some ample Cartier divisor $H$ on $X$ and for all $c>0$ there exist numbers $j \in \N$ and $i \in \N$ such that $i>cj$ and
$$
H^0(X, (S^{i} \sF)/\mbox{\rm Tor} \otimes \sO_X(jH)) \neq 0.
$$ 
\end{definition}

\begin{remark}
For locally free sheaves, the Definitions \ref{definitionpseffvb}, \ref{def:reflexive}, \ref{definitionpseff} obviously coincide. 
Even for reflexive sheaves, Definition \ref{definitionpseff} is more restrictive than Definition \ref{def:reflexive}: 
in general $(S^{i} \sF)/\mbox{Tor}$ is not reflexive and has less global sections than its bidual, so we might have 
$$
H^0(X, (S^{i} \sF)/\mbox{Tor} \otimes \sO_X(jH)) = 0,
$$ 
although $\sF$ is pseudoeffective in the sense of Definition \ref{def:reflexive}. 
We thank C. Gachet for the following example:
\end{remark}

\begin{example} \label{examplenotstrongly}
Let 
$$
\C^2 \rightarrow X=\{ (x,y,z) \in \C^3 \ | \ xy-z^2=0 \}, \ (u,v) \ \mapsto \ (u^2, v^2, uv)
$$
be the double cover of the $A_1$-singularity, we identify the polynomial ring of $X$ to its image 
$\C[u^2, v^2, uv]$ in $\C[u,v]$. The invariant elements under the involution 
$$
\holom{j}{\C[u,v]}{\C[u,v]}, \ f(p) \mapsto - f(-p) 
$$
are exactly the odd polynomials, i.e. the polynomials that can be written as
$u f + v g$
with $f,g \in \C[u^2, v^2, uv]$. Denote this set by $\C[u,v]^{\Z_2}$, then $\C[u,v]^{\Z_2}$ has a natural structure of
$\C[u^2, v^2, uv]$-module that is reflexive, but not (locally) free. The tensor power $(\C[u,v]^{\Z_2})^{\otimes 2}$
is generated by $u^2, v^2, uv$, so it naturally embeds into $\C[u^2, v^2, uv]$.
Remembering that this ring is actually the function ring of the $A_1$-singularity, we see that
$(\C[u,v]^{\Z_2})^{\otimes 2}$ is isomorphic to the maximal ideal defining the origin.

Let now $\holom{q}{A}{X}$ be the quotient of an abelian surface under the involution $z \mapsto -z$, so $X$ is the singular Kummer surface. Since $S^{[2]} \Omega_X$ is globally generated, the sheaf of reflexive differentials $\Omega_X^{[1]}$
is pseudoeffective. Let us see that it is not strongly pseudoeffective:
we have $\Omega_X^{[1]} \simeq \sF \oplus \sF$, where $\sF$ is the sheaf of $\Z_2$-invariants for the natural action on $\Omega_A \simeq \sO_A d z_1 \oplus \sO_A d z_2 $. It is immediate to see that, near the fixed points, the $\Z_2$-action on  $\sO_A d z_l$ identifies to the action $j$ in the paragraph above. Thus, using the local computation, we see that
$$
\sF^{\otimes i} \simeq 
\left\{
\begin{array}{lll} 
\sI_{X_{\sing}}^{\frac{i}{2}} & \mbox{if } i \mbox{ even},
\\ 
\sI_{X_{\sing}}^{\frac{i-1}{2}} \otimes \sF & \mbox{if } i \mbox{ odd}.
 \end{array}
\right.
$$
Combined with Example \ref{exampleidealsheaf} this shows that $\sF$ is not strongly pseudoeffective.
\end{example}

In general it is not clear if one can check
strong pseudoeffectivity by looking at a tautological class on (a modification of) the projectivisation. However there is a natural construction in a special case:

\begin{setup} \label{setupideal} {\rm 
Let $\sF$ be a torsion-free sheaf on a normal projective variety $X$  such that
$$
\sF \simeq \sI_Z \otimes \sE
$$
where $\sI_Z$ is an ideal sheaf and $\sE$ is a locally free sheaf.

Let $\holom{\mu}{\hat X}{X}$ be the blow-up of the ideal sheaf $\sI_Z$, then $\hat X$ is a (not necessarily normal) variety \cite[II, Prop.7.16]{Har77}.
We denote by 
$$
\sO_{\hat X}(1):=\fibre{\mu}{\sI_Z}\sO_{\hat X} 
$$
the tautological sheaf on $\hat X$.
Recall that by the definition of the blow-up \cite[II, \S 3]{AT82} one has
\begin{equation} \label{pushblowup}
\mu_* (\sO_{\hat X}(i)) = \sI_Z^i \qquad \forall \ i \geq 0.
\end{equation}
Note also that if $Z$ is locally generated by a regular sequence, one has 
$S^i \sI_Z \simeq \sI_Z^{i}$ for all $i \geq 0$ (e.g. \cite[Prop.2.2.8]{BC18}). In particular 
the blowup $\mbox{Bl}_{\sI_Z}(X)$ coincides with the projectivisation $\PP(\sI_Z)$.
}
\end{setup}

\begin{lemma} \label{lemmablowupstronglypseff}
In the situation of Setup \ref{setupideal}, the torsion-free sheaf $\sF$ is strongly pseudoeffective if and only if the locally free sheaf $\sO_{\hat X}(1) \otimes \mu^* \sE$ on the variety $\hat X$  is pseudoeffective.
\end{lemma}

\begin{proof} 
Let $H$ be an ample Cartier divisor on $X$.
By the projection formula and \eqref{pushblowup} one has
$$
\mu_* (\mu^* (\sO_X(jH) \otimes S^i \sE) \otimes \sO_{\hat X}(i))
\simeq 
\sO_X(jH) \otimes S^i \sE \otimes \sI_Z^i 
$$
for all $i \geq 0$. Moreover we know by \cite{Mic64} that
$$
(S^i \sI_Z)/\mbox{Tor} \simeq \sI_Z^{i}.
$$
Thus we obtain that
$$
H^0(\hat X, S^i(\sO_{\hat X}(1) \otimes \mu^* \sE) \otimes \sO_{\hat X}(j \mu^* H))
\simeq
H^0(X, (S^{i} \sF)/\mbox{Tor} \otimes \sO_X(jH))
$$
for all $i \geq 0$. Now we apply \cite[Lemma 2.7]{Dru18} (see \cite[Lemma 2.2.]{HLS20} for the case where the divisor is only big).
\end{proof}

\begin{example} \label{exampleidealsheaf}
Let $X$ be a normal projective variety, and let $I_Z \subset \sO_X$ an ideal sheaf. Then $\sI_Z$ is not strongly pseudoeffective by Lemma \ref{lemmablowupstronglypseff}.

On the other hand let $Z \subset \PP^2$ be a point, then $\sI_Z \otimes \sO_{\PP^2}(1)$
is strongly pseudoeffective. Indeed the locally free sheaf
$\sO_{\hat X}(1) \otimes \mu^* \sO_{\PP^2}(1)$
on the blowup $\hat X \simeq \mathbb F_1$ is pseudoeffective, since it is isomorphic
to $\sO_{\mathbb F_1}(F)$ where $F$ is the strict transform of a line through $Z$.
\end{example}

\begin{corollary} \label{corollarypsefftautological}
In the situation of Setup \ref{setupideal}, suppose that the ideal sheaf $\sI_Z$ is locally generated by a regular sequence (e.g. if $Z$ is a locally complete intersection scheme). Then the following statements are equivalent:
\begin{enumerate}
\item The sheaf $\sF$ is  strongly pseudoeffective;
\item The locally free sheaf $\sO_{\hat X}(1) \otimes \mu^* \sE$ on the blow-up 
$$
\mu: \hat X = \mbox{Bl}_{\sI_Z} X \rightarrow X,
$$
  is pseudoeffective; 
\item The tautological line bundle $\sO_{\PP(\sF)}(1)$ on the projectivisation $\PP(\sF)$ is pseudoeffective.
\end{enumerate}
\end{corollary}

\begin{proof}
The equivalence between a) and b) is shown in Lemma \ref{lemmablowupstronglypseff}. 

Denote by $\holom{\hat \pi}{\PP(\mu^* \sE)}{\hat X}$ the projectivisation, and by $\sO_{\PP(\mu^* \sE)}(1)$ 
its tautological sheaf. Since  $\sO_{\hat X}(1)$ is $\mu$-ample
and $\sO_{\PP(\mu^* \sE)}(1)$ is $\hat \pi$-ample and $\mu \circ \hat \pi$-nef, 
we know that $\hat \pi^*(\sO_{\hat X}(1)) \otimes \sO_{\PP(\mu^* \sE)}(1)$
is $\mu \circ \hat \pi$-ample. Since $\sI_Z$ is locally generated by a regular sequence, we have $S^i \sI_Z \simeq \sI_Z^{i}$ for all $i \geq 0$
(e.g. \cite[Prop.2.2.8]{BC18}). 
Thus for all $i \geq 0$ we have
$$
(\mu \circ \hat \pi)_* ((\hat \pi)^* (\sO_{\hat X}(i)) \otimes \sO_{\PP(\mu^* \sE)}(i))
\simeq 
\mu_* (\sO_{\hat X}(i) \otimes \mu^* S^i \sE)
\simeq
\sI_Z^{i} \otimes S^i \sE
\simeq S^i (\sI_Z \otimes \sE).
$$
Therefore we have an isomorphism $\psi:  \PP(\mu^* \sE) \rightarrow \PP(\sF)$
such that $\psi^* \sO_{\PP(\sF)}(1) = \hat \pi^*(\sO_{\hat X}(1)) \otimes \sO_{\PP(\mu^* \sE)}(1)$.
Thus b) and c) are equivalent.
\end{proof}

\subsection{Generalised Kodaira-Iitaka dimensions and functoriality} 

In this subsection we introduce a Kodaira-Iitaka dimension for reflexive sheaves and establish functoriality. 

\begin{definition}  \label{def:Kodaira-1} Let $X$ be a normal projective variety, and let $\sF$ be a reflexive sheaf on $X$. 
Let $\holom{\pi}{P}{X}$ be a desingularization of the normalization of
the unique component $\mathbb P'(\sF)$  of $\mathbb P(\sF)$ dominating $X$ such that 
the preimage of the singular locus of $X$ and of the singular locus of the sheaf $\sF$ is a divisor in $P$.  Let $\zeta$ be a tautological class on $P$, \cite[Defn.2.2]{HP19}. 
Then we define
$$
\kappa (X, \sF) = \kappa (P, \zeta).
$$
\end{definition} 

By construction $\kappa (P, \zeta) \geq 0$ if and only if $H^0(X,S^{[m]}(\sF)) \neq 0$ for some $m \in \N$.

\begin{remark*}
If $\pi: P \to X$ denotes the projection, then 
$$ H^0(P, \sO_P(m\zeta)) = H^0(X,S^{[m]}(\sF)) $$
for all positive numbers $m$, and therefore the definition is independent on the choices made. 
\end{remark*}

We will use Definition \ref{def:Kodaira-1} in Section 4 to introduece a generalised Kodaira dimenion of $X$ (Definition \ref{def:Kodaira-2}).

The next result generalizes Lemma \ref{lemma:func} for finite morphisms.

 \begin{lemma} \label{lemmacover} Let $f: \widetilde X \to X$ be a finite morphism of normal projective varieties. Let $\sF$ be a reflexive sheaf on $X$. 
Then the following holds:
\begin{itemize}
\item  The reflexive pullback $f^{[*]}(\sF)$ is pseudoeffective if and only if $\sF$ is pseudoeffective. 
\item One has $\kappa(\widetilde X, f^{[*]}(\sF)) = \kappa(X, \sF)$.
\end{itemize} 
 \end{lemma}

 \begin{proof}  Let $\holom{\pi}{P}{X}$ be the projective manifold from Definition \ref{def:Kodaira-1},
 and $\zeta$ a tautological class on $P$. By the construction of $\zeta$ (see  \cite[Defn.2.2]{HP19}), we have
\begin{equation} \label{propertytautological}
 \pi_*(\sO_P(m \zeta)) \simeq S^{[m]}(\sF)
\end{equation}
 for all positive integers $m$. 
We introduce the fibre product
 $$ 
 \widetilde P := P \times_X \widetilde X.
 $$
 Let $\sigma: \widehat P \to \widetilde P$ be a desingularization, which is an isomorphism outside the singular locus of $\widetilde P$
 and set $\widehat \zeta := \sigma^* p_1^*(\zeta)$,
 where $p_1: \widetilde P \to P$ denotes the projection. 
 Let $\widetilde \pi: \widetilde P \to \widetilde X $ be the projection and $\widehat \pi := \widetilde \pi \circ \sigma$.

{\it Claim:} There exists a $\widehat \pi$-exceptional divisor $\widehat D$ on $\widehat P$, such that 
$$ 
\widehat \pi_*(\sO_{\widehat P}(m (\widehat \zeta + \widehat D))) \simeq S^{[m]}f^{[*]}(\sF)
$$ 
 for all $m \in \N$.

{\em Proof of the claim:}
By \cite[III.5.10.3]{Nak04} there exists $\widehat \pi$-exceptional divisor $\widehat D$ on $\widehat P$ such that
$$
\widehat \pi_*(\sO_{\widehat P}(m (\widehat \zeta + \widehat D)))
$$
is reflexive for all $m \in \N$.

Let $X_0 \subset X$ be the locus where $X$ is smooth and $\sF$ is locally free. Since $X$ is normal and  $\sF$ is reflexive, the complement of $X_0$ has codimension at least two.
Set now $\widetilde X_0=\fibre{f}{X_0}$. Since $X_0$ is smooth, the restriction $f|_{\widetilde X_0}$ is flat.
Thus by flat base change \cite[III, Prop.9.3]{Har77} we have 
an isomorphism 
$$ 
\widetilde \pi_*(p_1^*(\sO_P(m\zeta)) \simeq f^{*} S^{[m]}(\sF)
$$
over $\widetilde X_0$. Since $\sF$ is locally free on $X_0$, we have
$$
f^{*} S^{[m]}(\sF)  \simeq f^{[*]} S^{[m]}(\sF)
$$
over $\widetilde X_0$. 
Thus $\widehat \pi_*(\sO_{\widehat P}(m (\widehat \zeta + \widehat D)))$ and $f^{[*]} S^{[m]}(\sF)$ are isomorphic
over $\widetilde X_0$. 
Since they are both reflexive, they are isomorphic : indeed the complement of $\widetilde X_0$ has codimension at least two,
since $f$ is finite.
For the same reason we have
$$  
f^{[*]} S^{[m]}(\sF) \simeq S^{[m]}f^{[*]}(\sF),
$$
which shows the claim.

{\em Proof of the first statement:}
If $\sF$ is pseudoeffective, fix an ample Cartier divisor $AH$ on $X$. 
Since $f$ is finite, a section of $S^{[i]} \sF \otimes \sO_X(jH)$ pulls back to a section
$S^{[i]} (f^{[*]}(\sF)) \otimes \sO_X(j f^* H)$. Thus $f^{[*]}(\sF)$ is pseudoeffective, by Definition \ref{def:reflexive}.

Assume now that $f^{[*]}(\sF)$ is pseudoeffective,  so the divisor class $\widehat \zeta + \widehat D$ is pseudoeffective by \cite[Lemma 2.3]{HP19}.
Let now $\holom{\mu}{\hat P}{P'}$ and $\holom{p_1'}{P'}{P}$ be the Stein factorisation of the generically finite morphism $p_1 \circ \sigma$, i.e. $\mu$ is birational onto the normal variety $P'$ and $p_1'$ is finite.
Then $\widehat \zeta = (p_1 \circ \sigma)^* \zeta = (p_1' \circ \mu)^* \zeta$,
so 
$$
\mu_* (\widehat \zeta + \widehat D) = (p_1')^* \zeta + \mu_*(\widehat D)
$$
is a pseudoeffective Weil divisor class (cf. \cite[II, Defn.5.5]{Nak04}). 
Setting $D= (p_1')_* \mu_*(\widehat D)$, we have an inclusion of Weil divisors
$\mu_*(\widehat D) \subset (p_1')^* D$. Thus we have an inclusion of Weil divisor classes
$$
(p_1')^* \zeta + \mu_*(\widehat D)
\subset 
(p_1')^* (\zeta +D), 
$$
which shows that $(p_1')^* (\zeta +D)$ is pseudoeffective. Since $\zeta+D$ is Cartier, this shows that $\zeta+D$ is pseudoeffective. 
Since $\widehat D$ is $\widehat \pi$-exceptional, the effective divisor $D$ is $\pi$-exceptional.
By \eqref{propertytautological} we  thus have
$$
\pi_*( \sO_P(m (\zeta + D)))^{**} =  \pi_*( \sO_P(m \zeta)^{**} \simeq S^{[m]}(\sF)
$$
for all $m \in \N$. This shows that
$$ 
S^{[m]}(\sF) \simeq  \pi_*( \sO_P(m \zeta)  \hookrightarrow 
\pi_*( \sO_P(m (\zeta + D))) \hookrightarrow \pi_*( \sO_P(m (\zeta + D)))^{**}  \simeq S^{[m]}(\sF)
$$ 
is a chain of isomorphisms for all $m \in \N$. Hence the reflexive sheaf $\sF$ is pseudoeffective by \cite[Lemma 2.3]{HP19}. 

{\em Proof of the second statement:}
As for the first statement, the inequality $\kappa(\widetilde X, f^{[*]}(\sF)) \geq \kappa(X, \sF)$ is immediate.
Let us show the other inequality: by the claim we know that $\widehat \zeta + \widehat D$ is a tautological class on $\hat P$.
Thus by assumption, one has
 $$ 
 \kappa(\hat P, \widehat \zeta + \widehat D) = \kappa((\widetilde X, f^{[*]}(\sF))) \geq 0.
 $$ 
Hence $\kappa (P', (p_1')^*(\zeta) + \mu_*(\widehat D)) =  \kappa((\widetilde X, f^{[*]}(\sF))) \geq 0$, and,
 $$ 
 \kappa(P', (p_1')^*(\zeta + D)) \geq \kappa (P', (p_1')^*(\zeta) + \mu_*(\widehat D)),
 $$
 where we use again the inclusion $(p_1')^* \zeta + \mu_*(\widehat D)
\subset (p_1')^* (\zeta +D)$. 
Since 
 $$ 
 \kappa(P', (p_1')^*(\zeta + D)) = \kappa(P, \zeta + D)),
 $$
 by \cite[Thm 5.13]{Ue75}, the statement follows.
 \end{proof} 

In Section 4 this will be applied to sheaves of reflexive differentials, Corollary {propcover2}. 

\subsection{Pseudoeffective sheaves on fibered surfaces} \label{subsectionfibered}

The results of this section will be relevant to the study of elliptic surfaces. Let us recall the Zariski decomposition on surfaces \cite[Thm.]{Bau09}, \cite[Thm.2.3.19]{Laz04a}:
\begin{itemize}
\item Let $D$ be an effective $\Q$-divisor on a smooth surface. Then there exist uniquely determined effective $\Q$-divisor $P$ and $N$ with
$$
D = P + N
$$ 
such that $P$ is nef, the divisor $N= \sum a_j N_j$ is zero or has negative definite
intersection matrix and $P \cdot N_j=0$ for all $j$.
\item The same statement holds if $D$ is pseudoeffective, except that in this case $P$ is not necessarily effective.
\end{itemize}

The following lemma is well-known to experts, we give the details in order to prepare the proof of its singular version in Lemma \ref{lemS2}.

\begin{lemma} \label{lemS1}  Let $X$ be a smooth projective surface, and let $f: X \to B$ be a fibration over a smooth curve $B$. 
Let $L$ be a pseudoeffective line bundle on $X$ such that $L_F \simeq \sO_F$ for the general fiber $F$ of $f$. 
Let $D$ be a Cartier divisor such that $L \simeq \sO_X(D)$, and let 
$$
D = P + N
$$
its Zariski decomposition. Then the following holds:
\begin{enumerate}
\item Up to taking multiples, one has $\sO_X(P) \simeq f^* M$ with $M$ a nef line bundle.
\item If $P \not\equiv 0$, one has $\kappa(L) \geq 1$. 
\item If $P \equiv 0$, there exists $m \in \N$ and
a numerically trivial line bundle $M$ on $B$ such that $\kappa (X, L^{\otimes m} \otimes f^* M^*) = 0$.
\end{enumerate}
\end{lemma}

\begin{proof}  
Note that the statements are invariant under taking multiples.

Up to replacing $L$ by some multiple, we can assume that $P$ and $N$ are Cartier divisors.
Since  $N$ is effective and
$$
\sO_F \simeq \sO_F(D) \simeq \sO_F(P) \otimes \sO_F(N)
$$
we see that $N$ has support in the fibres of $f$ and $\sO_F(N) \simeq \sO_F(P) \simeq \sO_F$. Thus the direct image sheaf
$f_* (\sO_X(P))$ is locally free of rank one, and we have
$$
\sO_X(P) = f^* (f_* \sO_X(P)) \otimes \sO_X(E)
$$
where $E$ is an effective divisor $E$, supported in fibers of $f$.
Since $P^2 \geq 0$ and $E^2 \leq 0$ \cite[III, Lemma 8.2]{BHPV04}, it follows $E^2 = 0$. Hence by (ibid) there exists a number $k$ such that $kE = f^*(H)$ with some effective divisor $H$. 
Thus, again up to replacing $L$ by some multiple, we have  $\sO_X(P) = f^* M$ for some line bundle $M$ on $B$. 

If $M \not\equiv 0$, it is ample on $B$, so 
$$
1= \kappa(M) = \kappa(P)= \kappa(L).
$$
If $P \equiv 0$, then $M \equiv 0$, so $\kappa(L \otimes f^* M^*) = \kappa(N)=0$.
\end{proof}

\begin{lemma} \label{lemS2}  
Let $\hat X$ be an irreducible reduced projective surface, and let $\hat f: \hat X \to B$ be a fibration over a smooth curve $B$
such that the general fibre $F$ is smooth.
Let $\hat L$ be a pseudoeffective line bundle on $\hat X$ such that $\hat L|_F \simeq \sO_F$. 
\begin{enumerate}
\item There exists $m \in \N$ and a numerically trivial line bundle $M$ on $B$ such that 
$h^0(\hat X, \hat L^{\otimes m} \otimes \hat f^* M) \geq 0$.
\item If $h^0 (\hat X, \hat L^{\otimes m} \otimes \hat f^* M) > 1$ for some numerically trivial line bundle $M$, then $h^0 (\hat X, \hat L^{\otimes k}) > 1$ for some $k \in \N$.
\end{enumerate}
\end{lemma} 

\begin{remark} \label{remarkrationalcurve}
We will frequently apply the lemma in the case where $B$ is a rational curve. In this case one  obtains $h^0 (\hat X, \hat L^{\otimes m}) > 0$ for some $m \in \N$.
\end{remark}

Since the dimension of the space of global sections is not necessarily invariant under normalisation, the statement requires some work:

\begin{proof}  
Note that the statement is invariant under taking multiples.

Let $\holom{\mu}{X}{\hat X}$ be the composition of normalisation and desingularisation, set $f := \hat f \circ \mu$.
Then $\mu^* \hat L =: \sO_X(D)$ is pseudoeffective, and we denote by
$$
D = P + N
$$
the Zariski decomposition. By Lemma \ref{lemS1}, a) we have, up to taking multiples, that $\sO_X(P) \simeq f^* M$
for some line bundle on $M$. Thus we see that $\hat L \simeq \hat f^* M \otimes \hat N$, where
$\hat N$ is a Cartier divisor on $\hat X$ such that $N \equiv \mu^* \hat N$. 
By Lemma \ref{lemS1}, b),c) it is sufficient to show that we can choose $\hat N$ to be effective.

{\em Proof of the claim.} 
Since 
$$
\hat N \simeq \hat L \otimes (\hat f)^* M^*,
$$
we have $\sO_F(\hat N) \simeq \sO_F$. Thus $(\hat f)_* \sO_{\hat X}(\hat N)$ is locally free of rank one, hence for some $m \gg 0$, we have
$$
H^0(\hat X, \sO_{\hat X}(mF) \otimes \sO_{\hat X}(\hat N)) 
\simeq
H^0(B, \sO_B(m) \otimes (\hat f)_* \sO_{\hat X}(\hat N)) \neq 0.
$$
Thus we can fix an effective Cartier divisor $E$ on $\hat X$ such that $\sO_{\hat X}(E) \simeq \sO_X(mF) \otimes \sO_{\hat X}(\hat N)$.
We can decompose
$$
E = f^* E_B + R
$$
where $E_B$ is an effective $\Q$-divisor on $B$ and 
$R$ is an effective divisor such that for every connected component $R' \subset R$ we have a strict set-theoretical inclusion
$R' \subsetneq \fibre{\hat f}{\hat f(R')}$.
Up to taking multiples and replacing $E_B$ by a linearly equivalent divisor, 
we can also suppose that $supp (f^* E_B) \subset X_{nons}$. 
Now observe that 
$$
\mu^* E = \mu^* f^* E_B + \mu^* R
$$
is a Zariski decomposition. Since
$$
\mu^* E \equiv m F + N. 
$$
is also a Zariski decomposition 
and the negative part of the Zariski decomposition is unique in the numerical equivalence class, we finally obtain $N = \mu^* R$.
\end{proof}

\begin{corollary} \label{corpseffkappa}  Let $X$ be a smooth projective surface, and let $f: X \to B$ a fibration over a smooth {\em rational} curve $B$. Let $Z \subset X$ be a local complete
intersection of codimension $2$. Let $L$ be a line bundle on $X$ such that $L_F \simeq \sO_F$ for the general fiber $F$ of $f$. 
Then $\sI_Z \otimes L$ is  strongly pseudoeffective if and only if $\kappa (X, \sI_Z \otimes L) \geq 0$, i.e., if there exists a positive integer $m$ such that 
$$ 
H^0(X,\sI_Z^m \otimes L^m) \ne 0.
$$ 
\end{corollary}

\begin{proof} One direction being obvious, so assume that $\sI_Z \otimes L$ is pseudoeffective. 
Let $$\mu: \hat X \to X$$
be the blow-up of $X$ along $Z$ and denote by $E$ the exceptional divisor.
Since $Z$ does not surject onto $B$, the general fibre of $f \circ \mu$ is smooth.
 
By Corollary \ref{corollarypsefftautological} the 
the line bundle $\sO_{\hat X}(1) \otimes \mu^*(L) \simeq \sO_{\hat X}(-E) \otimes \mu^*(L)$ is pseudoeffective. 
By Remark \ref{remarkrationalcurve} we see that there exists some $m \in \N$ such that
$$
H^0(\hat X, \sO_{\hat X}(-mE) \otimes \mu^*(L^{\otimes m})) \neq 0.
$$
Since $Z$ is a local complete intersection, we have $\mu_*(\sO_{\hat X}(-mE)) = \sI_Z^m$. Thus we conclude by the projection formula.

\end{proof}

\section{Pseudoeffective cotangent sheaves and the Nonvanishing Conjecture} 

In this section we gather some basic facts on the behaviour of pseudoeffective cotangent sheaves under 
birational maps and finite covers. We also establish a relation with the MRC fibration of the variety.

\begin{proposition}  \label{prop:bir}
Let $\mu: \hat X \to X$ be a birational morphism of normal projective varieties.
\begin{enumerate}
\item If $\Omega^{[q]}_{\widehat X}$ is pseudoeffective, so does $\Omega^{[q]}_X$.
\item Suppose that  $X$ is smooth. Then the converse also holds.
\end{enumerate}
\end{proposition} 

\begin{proof} a) We choose ample divisors $\hat H$ on $\hat X$ and $H$ on $X$ such that 
$$ 
\hat H  + E = \mu^*(H) 
$$
with $E$ an effective divisor supported on the exceptional locus of $\mu$. 
By assumption, for all $c > 0$, there are numbers $i$ and $j$ with $i > cj$ such that 
$$ 
H^0(\hat X, S^{[i]}\Omega_{\hat X}^{[q]} \otimes \sO_{\hat X}(j\hat H))\ne 0.
$$
In particular, 
$$
0 \ne H^0(\hat X, S^{[i]}\Omega_{\hat X}^{[q]} \otimes \sO_{\hat X}(j \mu^*H)) = H^0(X,\mu_*(S^{[i]}\Omega_{\hat X}^{[q]}) \otimes \sO_X(jH)).
$$
Since $\mu_*(S^{[i]}\Omega_{\hat X}^{[q]}) \subset S^{[i]}\Omega^{[q]}_X,$ we conclude. 

b) Suppose that $X$ is smooth. By a), we may assume $\hat X$ to be smooth as well. 
Then all the involved sheaves are locally free, in particular
$$ 
S^{i}\mu^*(\Omega^q_X)  = \mu^*(S^{i} \Omega^q_X) \subset S^{i}\Omega^q_{\hat X},
$$
and the claim follows. 
\end{proof} 

\begin{example}  Assertion \ref{prop:bir},b) fails in general if $X$ is singular, even if $X$ has only canonical singularities. In fact, the paper \cite{GKP16b} constructs a normal projective 
surface $X$ with the following properties.
\begin{itemize}
\item $X$ has only ADE singularities ;
\item the minimal desingularization  $\hat X$ is rationally connected ;
\item $H^0(X,S^{[2]}\Omega^1_X) \ne 0$.
\end{itemize} 
Thus $\Omega^1_{\hat X}$ is 
not pseudoeffective in contrast to $\Omega^{[1]}_X$. 

Another example is a K3 surface of Kummer type, see Example \ref{exampleproducttype}. 
\end{example} 

\begin{corollary} \label{cor:min} Let $X$ be a normal projective variety with klt singularities, 
and let $X \dasharrow X'$ be a composition of divisorial contractions and flips. 
If $\Omega^{[q]}_{X}$ is pseudoeffective, so does $\Omega^{[q]}_{X'}$.
\end{corollary} 

\begin{proof} By Proposition \ref{prop:bir},a) it suffices to treat the case of a flip $\mu: X \dasharrow X'$. 
Since a flip is an isomorphism in codimension two, one has
$$
H^0(X, S^{[i]} \Omega_X^{[q]} \otimes \sO_X(j H)) \simeq
H^0(X', S^{[i]} \Omega_{X'}^{[q]} \otimes \sO_{X'}(j \mu_* H))
$$
for all $i,j \in \N$. Thus the condition in Definition \ref{def:reflexive} holds for a big $\Q$-Cartier divisor, which is sufficient
(see \cite[Lemma 2.2]{HLS20}).
\end{proof}

\begin{proposition} \label{prop:cover}
Let $f: \widetilde X \to X$ be a finite surjective morphism of normal projective varieties. If $\Omega^{[q]}_X$ is pseudoeffective, so does $\Omega^{[q]}_{\widetilde X}$. 
If $f$ is quasi-\'etale, the converse also holds. 
\end{proposition}

\begin{proof}  
Over the smooth locus of $X$ we have a injective morphism
$$ 
f^*(S^i \Omega^q_{X_{\nons}}) \to S^{[i]} \Omega^{[q]}_{\widetilde X}.
$$
Since the complement of \fibre{f}{X_{\nons}} has codimension at least two, the morphism  extends to 
$$ 
f^{[*]}(S^{[i]} \Omega^{[q]}_X) \to S^{[i]}(\Omega^{[q]}_{\widetilde X}).
$$
which gives the first claim.

Assume now that $f$ is quasi-\'etale and that $\Omega^{[q]}_{\widetilde X}$ is pseudoeffective. Then $f^{[*]}(\Omega^{[q]}_X) \simeq \Omega^{[q]}_{\widetilde X}$ is pseudoeffective. 
Now we conclude by Lemma \ref{lemmacover}.
\end{proof} 

At this point we introduce generalised Kodaira dimension:

\begin{definition}  \label{def:Kodaira-2} 
Let $X$ be a normal projective variety with klt singularities and $1 \leq q \leq n$.
Then we define
$$ 
\kappa_q(X) = \kappa (X, \Omega_X^{[q]}). 
$$
\end{definition}

In case $q = \dim X$, we have of course $\kappa_q(X) =\kappa(X)$.

 As a corollary to Lemma \ref{lemmacover}  we  obtain a generalisation of \cite[Prop. 2.2]{Ane18}:  

\begin{proposition} \label{propcover2} Let  $\widetilde X \to X$ be a quasi-\'etale morphism of  projective varieties with klt singularities. 
Then 
$$ \kappa_q(\widetilde X) = \kappa_q(X)$$
for all $1 \leq q \leq \dim X$. 
\end{proposition} 

\begin{proof} 
This follows from Lemma \ref{lemmacover}, since for a quasi-\'etale morphism
$$
 f^{[*]}(\Omega_X^{[q]}) =  \Omega_{\widetilde X}^{[q]}.
$$
\end{proof} 

\begin{remark} Let $\mu: \widehat X \to X$ be a birational morphism of normal projective varieties with klt singularities. 
Then $\kappa_q(\widehat X) \leq \kappa_q(X) $ with equality if $X$ is smooth. 
The same inequality holds if $\mu: \widehat X \dasharrow X$ is a composition of divisorial contractions and 
flips.
\end{remark}

Although Conjecture \ref{conj:nv} can be formulated for any $p$, we are mainly interested in the case $p = 1$. We next confirm the conjecture for $p = 1$ in case $K_X \equiv 0$.

\begin{proposition} \label{prop:pb1} 
Let $X$ be a normal projective variety with klt singularities such that $K_X \equiv 0$. Assume that $X$ is smooth in codimension two, e.g., $X$ has terminal singularities. Then the following are equivalent:
\begin{enumerate}
\item $\Omega^{[1]}_X$ is  pseudoeffective ;
\item we have $\widetilde q(X) > 0$, i.e., there exists a quasi-\'etale cover $\widetilde X \to X$ such that $H^0(\widetilde X, \Omega^{[1]}_{\widetilde X}) \ne 0$ ;
\item we have $H^0(X,S^{[m]}\Omega^1_X)) \ne 0$ for some positive integer $m$, i.e., $\kappa_1(X) \geq 0$. 
\end{enumerate}

\end{proposition} 

\begin{proof} 
By \cite[Thm.1.6]{HP19} we know that 1) implies 2). 
By Proposition \ref{propcover2} we know that 2) implies 3) which obviously implies 1).
\end{proof}

We will now discuss the relation between the 
pseudoeffectivity of $\Omega^q_X$ and the MRC fibration. 
First, the rational connectedness criterion given in \cite{CDP12} can be stated as follows.

\begin{theorem} \label{thm:rc} Let $X$ be a projective manifold of dimension $n$. Then $X$ is rationally connected if and only if  
for all $1 \leq q \leq n$ the vector bundle $\Omega^q_X$ is not pseudoeffective. 
\end{theorem}

\begin{proof} If $X$ is rationally connected, it is dominated by very free rational curves \cite[IV.3.9]{Ko96}. 
It is then straightforward to verify the vanishing condition in Definition \ref{def:reflexive}.

Assume that $\Omega^q_X$ is not pseudoeffective for all $1 \leq q \leq n$.
Let $\sF \subset \Omega^q_X$ be an invertible subsheaf,
then $\sF$ is not pseudoeffective (see Example \ref{examplesubsheaf}). Hence by \cite[Thm.1.1]{CDP12}, $X$ is rationally connected. 
\end{proof} 

Theorem \ref{thm:rc} can be generalised as follows. 

\begin{theorem} \label{thm:MRC} Let $X$ be a projective variety of dimension $n$. 
Fix some $r \in \{1, \ldots, n\}$ and assume that $\Omega^{[q]}_X$ is not pseudoeffective
for all $r \leq q \leq n$. 

Then $X$ is uniruled, and the base $Z$ of the MRC fibration satifies $\dim Z \leq r-1$. 
\end{theorem} 

\begin{proof} 
By Proposition \ref{prop:bir} we can assume without loss of generality that $X$ is smooth.
Since $K_X = \Omega^n_X$ is not pseudoeffective, the manifold $X$ is uniruled by \cite{BDPP13}. 
Hence we consider the MRC fibration 
$$
f: X \dasharrow Z.
$$ 
Up to replacing $Z$ by a resolution and $X$ by a blow-up, we may assume, by Proposition \ref{prop:bir},
that $Z$ is smooth and $f$ is a morphism.

Arguing by contradiction we suppose that $d := \dim Z \geq r$. By \cite{GHS03}, 
the variety $Z$ is not uniruled, hence $K_Z$ is pseudoeffective. Choose $\sH_Z$ ample on $Z$ and set $\sH_X = f^*(\sH_Z)$. 
Since $K_Z$ is pseudoeffective, for all  $c > 0$ there exist integers $i$ and $j$ with  $i>cj$ such that 
$$ 
H^0(Z,\sO_Z(iK_Z) \otimes \sO_Z(j\sH_Z)) \ne 0.
$$ 
Since $0 \neq f^*(\Omega^d_Z) \subset \Omega^d_X$, 
it follows that 
$$ 
H^0(X,S^i\Omega^d_X \otimes \sO_X (j\sH_X)) \ne 0.
$$
Thus $\Omega^d_X$ is pseudoeffective and $d \geq r$, a contradiction to our assumption. 
\end{proof} 

If $X$ is smooth, the converse to Theorem \ref{thm:MRC} is also true:

\begin{proposition} \label{propositionuniruled}
Let $X$ be a uniruled projective manifold of dimension $n$,  and let $f: X \dasharrow Z$ be the MRC fibration. 
If $d =\dim Z$, then for every  $\Omega^q_X$ is not pseudoeffective for all $d+1 \leq q \leq n$.
\end{proposition}

\begin{proof} 
Let $F$ be a general fiber of $f$. Then $F$ is rationally connected of dimension 
$\dim F = n-d$. Let $C \subset F$ be a general very free rational curve \cite[IV.3.9]{Ko96}, so ${T_F}_{\vert C}$ is ample.
Then $$
{T_X}|_{C} \simeq \sO_C^{\oplus d} \oplus {T_F}|_{C}.
$$
Thus for every $q \geq d+1$, the exterior power $\bigwedge^q {T_X}|_{C}$ is ample. Hence $\Omega^q_X$ is not pseudoeffective. 
\end{proof} 

For smooth varieties, the MRC-fibration  should allow to reduce Conjecture \ref{conj:nv} to non-uniruled varieties:

\begin{conjecture} \label{conj:uniruled} 
Let $X$ be a uniruled projective manifold, and let $f: X \dasharrow Z$ be the MRC fibration to the projective manifold
$Z$. Let $1 \leq q \leq n$.  Then $\Omega^q_X$ is pseudoeffective if and only if $\Omega^q_Z$ is pseudoeffective. 
\end{conjecture}

Note that, by Proposition \ref{prop:bir}, we may assume $f$ holomorphic. 
Then one direction is clear:  if $\Omega^q_Z$ is pseudoeffective, then so does $f^*(\Omega^q_Z)$ by Lemma \ref{lemma:func}. Hence $\Omega^q_X$ is pseudoeffective, see Example \ref{examplesubsheaf}. 
Vice versa assume that $\Omega_X^q$ is pseudoeffective. 
Applying  \cite[Cor.1.3.2]{BCHM10}, $f$ factors into a sequence of divisorial contractions and flips, ending with a Mori fiber space 
 $f': X' \to Z$ of relative Picard number one. By  Corollary \ref{cor:min}, we may therefore assume that $f$ is a Mori contraction, but now $X$ may have terminal singularities instead of being smooth. 

\begin{proposition} \label{pro:uniruled} 
Conjecture \ref{conj:uniruled} is true in dimension three.
\end{proposition} 

\begin{proof} As just noticed it suffices to treat Mori contractions $f: X \to Z$ where $X$ has terminal singularities. 
Since $Z$ is not uniruled, we may assume that $\dim Z = 2$; otherwise $Z$ is a curve of genus $g \geq 1$ and there is nothing to prove. 
Further, since $K_Z$ is pseudoeffective, only the case $q = 1$ needs to be treated. Now
 $f$ is a generically a conic bundle \cite[4.1]{AW95}. More precisely, the singular locus of $X$ being finite, 
there is a finite set $B = f({\rm Sing}(X))$ in $Z$ such that, setting $Z_0 = Z \setminus B$ and $X_0 = f^{-1}(Z_0)$,
the map $f_0: X_0 \to Z_0$ is a conic bundle with only finitely many non-reduced fibers. Furthermore, $f^{-1}(B)$ is one-dimensional. 

Since $-K_X$ is relatively ample and relatively globally generated on $X_0$, 
we can choose a very ample Cartier divisor $H$ on $Z$ such that $-K_{X/Z} + f^* H=: A$ is ample and satisfies $H^0(X, \sO_X(A)) \neq 0$.
In particular there is an injection 
\begin{equation} \label{injectomega}
\omega_{X/Z} \hookrightarrow \sO_X(f^* H).
\end{equation}
{\it We claim} that for every $c>1$ there exist positive integers $k, j$ such that $k \geq cj$ such that
$$
H^0(X_0, f^* S^k \Omega_Z \otimes \sO_X(j (A+f^* H))) \neq 0.
$$
Since $f^{-1}(B)$ has codimension at least two, this shows that $f^* \Omega_Z$ is pseudoeffective. Thus $\Omega_Z$ is 
pseudoeffective by Lemma \ref{lemma:func}.

{\em Proof of the claim.} Since $\Omega_X$ is pseudoeffective, there  exist positive integers $i,j$ such that $i \geq 2cj$ such that
$$
H^0(X_0, S^i \Omega_X \otimes \sO_X(j A)) \neq 0.
$$
We consider the canonical exact sequence
 \begin{equation} \label{EQ}  
 0 \to f^*(\Omega_{Z_0}) \buildrel {df} \over {\to} \Omega_{X_0} \to \Omega_{X_0/Z_0} \to 0.
 \end{equation} 
Since $X_0 \rightarrow Z_0$ is a conic bundle, we know that $df$ cannot vanish along a divisor $D$. Thus $\Omega_{X/Z}$ is torsion free and the singular locus of $ \Omega_{X/Z}$ is at most one-dimensional. 
Thus we get
$$
H^0(X_0, f^* S^k \Omega_{Z_0} \otimes \omega_{X_0/Z_0}^{[i-k]} \otimes \sO_X(j A)) \neq 0.
$$ 
for some $k \in \{ 0, \ldots, i \}$.
Since $\omega_{X_0/Z_0}^{[i-k]} \otimes \sO_X(j A)$ has negative degree on the fibres of $f$ if $i-k>j$ we see that
$k \in \{ i-j, \ldots, j\}$. Note that since $i \geq 2cj$ and $c>1$ this implies that $k \geq j$.  
Using the morphism \eqref{injectomega} obtain
$$
H^0(X_0, f^* S^k \Omega_{Z_0} \otimes \sO_X(j A+(i-k)f^* H)) \neq 0.
$$
Since $i-k \leq j$ we finally obtain the claim.
\end{proof} 

\begin{remark*}
The key point of the proof above is that the morphism $df$ does not vanish along a divisor $D_0$.
In higher dimension, since the total space of the Mori fibre space is not necessarily smooth, this might very well happen.
Then
these type of arguments only show that $f^*(\Omega^q(D)$ is pseudoeffective where $D$ has support inside the support of $D_0$. 
At least in dimension three we also see that 
$$ H^0(X,S^m\Omega_X) = H^0(Z, S^m\Omega_Z)$$
where $X \dasharrow Z $ is the MRC fibration of a uniruled smooth threefold $X$ and $Z$ is smooth. 
\end{remark*}

\begin{corollary} \label{corollarysurfaces}
Let $X$ be a smooth projective surface such that $\Omega_X$ is pseudoeffective. If $\kappa (X) \leq 0$, then 
Conjecture \ref{conj:nv} holds. 
\end{corollary} 

\begin{proof}
If $\kappa(X)=-\infty$, the surface $X$ is uniruled. Since $\Omega_X$ is pseudoeffective, by Proposition \ref{propositionuniruled} the base of the MRC fibration is a curve of genus at least one. Thus we have $q(X)>0$.

If $\kappa(X)=0$, let $X_{min}$ be the minimal model of $X$. Then $\Omega_{X_{min}}$ is pseudoeffective by Proposition \ref{prop:bir}.
Thus by Proposition \ref{prop:pb1} one has $H^0(X_{min},S^{m}\Omega^1_{X_{min}})) \ne 0$ for some positive integer $m$. 
Since $X_{min}$ is smooth, we have an isomorphism
$$
H^0(X,S^{m}\Omega^1_X))  \simeq H^0(X_{min},S^{m}\Omega^1_{X_{min}})).
$$
\end{proof}

In the next two sections we will deal with surfaces $X$ of Kodaira dimension $\kappa (X) = 1$.

\section{Elliptic surfaces: general set-up and the non-isotrivial case} \label{sectionelliptic}

We are starting here to study elliptic fibrations $f: X \to B$ with $\kappa (X) = 1$ towards Conjecture \ref{conj:nv}. If 
$f$ is almost smooth, i.e., the only singular fibers are multiples of elliptic curves, then $c_2(X) = e(X) = 0$ \cite[III, Prop.11.4]{BHPV04}. 
Thus by Noether's formula $\chi(X, \sO_X) \leq 0$, 
 and therefore $q(X) > 0$, so there
is nothing to prove. 
We first fix notations. 

\subsection{Elliptic fibrations: the setup}

\begin{setup} \label{setup} 
{\rm 
Let $X$ be a smooth projective surface, and let $\holom{f}{X}{B}$
be an elliptic fibration onto a smooth curve $B$.
We set 
$$
D = \sum_{b \in B} f^*b - (f^* b)_{red},
$$
so $D$ is an effective divisor having support exactly on the irreducible components of a fibre that are not reduced.
The exact sequence
$$
0 \rightarrow  f^*\Omega_ B  \to \Omega_X \to \Omega_{X/B} \rightarrow 0
$$
induces an exact sequence
\begin{equation} \label{sequencemain}
0 \to f^*\Omega_ B (D) \to \Omega_X \to \sI_Z \otimes \omega_{X/B}(-D) \to 0,
\end{equation} 
where $Z$ is a local complete intersection scheme of codimension two
whose support coincides with the singular points of the reduction $(f^* b)_{red}$ of
the fibres \cite[Prop.3.1(iii)]{Ser96}. 

We denote by $\holom{\pi}{\PP(\Omega_X)}{X}$ the projectivisation,
and by $\zeta \rightarrow \PP(\Omega_X)$ the tautological class.

We set 
\begin{equation} \label{defY}
Y:=\PP(I_Z \otimes  \omega_{X/B}(-D)) \subset \PP(\Omega_X).
\end{equation} 
Since $I_Z$ is a local complete intersection of codimension two, the projectivisation coincides with
the blow-up of the ideal sheaf $I_Z$ (see Setup \ref{setupideal}). In particular $Y$ is a prime divisor in $\PP(\Omega_X)$ and
\begin{equation} \label{formulaY}
[Y] = \zeta - \pi^* c_1(f^*\Omega_ B (D)).
\end{equation}

Denote by $K \subset B$ the finite set of points such that the fibre $f^* b$
is not multiple and not reduced. 
The divisor $D$ can be decomposed as 
\begin{equation} \label{divisorD}  
D = \sum_{i=1}^s (m_i-1) F_i + \sum_{b \in K} D_{0,b}, 
\end{equation} 
where the $F_i$ are the reductions of multiple $f$-fibres 
and 
$$
D_0 := \sum_{b \in K} D_{0,b}
$$ 
is simply the remainder, i.e. the part of $D$ coming from non-multiple, non-reduced fibres. It follows from Kodaira's classification \cite[V, Sect. 7, Table 3]{BHPV04} 
that the support of $D_0$ does not contain any fibre, so the intersection matrix
of $D_0$ is negative definite by Zariski's lemma \cite[III, Lemma 8.2]{BHPV04}.
It is now straightforward to check that \eqref{divisorD} is the Zariski decomposition of $D$ (see Subsection \ref{subsectionfibered})
with $P=\sum_{i=1}^s (m_i-1) F_i$.

If $f$ is relatively minimal, the canonical bundle formula \cite[V, Thm.12.1, Prop.12.2]{BHPV04} 
holds: 
\begin{equation} 
\label{cbf} 
\omega_X \simeq f^*(\omega_B \otimes (R^1 f_* \sO_X)^*) \otimes \sO_X(\sum_{i=1}^s (m_i-1)F_i), \end{equation} 
and $\deg (R^1 f_* \sO_X)^*=\chi(X, \sO_X)$.

Finally assume that $f$ is relatively minimal and $B = \PP^1$. Then \eqref{cbf}
implies that
\begin{equation} \label{formulaKXL}
K_{X/B}-D \sim a F -  D_0
\end{equation}
where $a=\chi(X, \sO_X)$ and $F$ is a general fiber.
}
\end{setup}

\begin{proposition} \label{propA}
In the situation of Setup \ref{setup}, suppose that
$f^*\Omega_ B (D)$ is pseudoeffective. Then we have $\tilde q(X) > 0$. 
\end{proposition} 

\begin{proof}  
The statement is trivial if $g(B) \geq 1$, so assume $B \simeq \PP^1$.
We follow the philosophy of \cite[Sect.3.5]{Cam04}.
By Remark \ref{remarkrationalcurve} one has $\kappa (f^*\Omega_ B (D)) \geq 0$. 
Choose a positive integer $m$ and a  non-zero section $s \in H^0(X, (f^*\Omega_ B (D))^{\otimes m})$. 
Let $E$ be the divisor defined by $s$. Note that $E$ is supported on fibers of $f$ and that 
$$ 
E \sim m(D-2F),
$$
where $F$ is a general fiber of $f$. 
By the discussion in Setup \ref{setup} we know that the nef part of $D$ is represented by
$\sum_{i=1}^s (m_i-1) F_i \equiv \lambda F$ where the $F_i$ are the reductions of multiple $f$-fibres. 
Since
$$ 
D \sim_{\mathbb Q} \frac{1}{m} E + 2F,
$$
is a decomposition in effective divisors and $2F$ is nef, we obtain $\lambda \geq 2$.

Introducing the $\mathbb Q$-divisor $\Delta = \sum_{i=1}^s (1-\frac{1}{m_i}) p_i$ with $p_i = f(F_i)$, 
we have $f^* \Delta = \sum_{i=1}^s (m_i-1) F_i$, so
$$
\sum_{i=1}^s (1-\frac{1}{m_i}) = \lambda \geq 2.
$$
Thus $f$ has at least three multiple fibers. 
Then there is a ramified base change $\tilde B \to B$ inducing an \'etale map $$\tilde X \to X$$
 such that $\tilde f: \tilde X \to \tilde B$ has no multiple fibers and $g(\tilde B) \geq 1$,
see e.g. \cite[IV.9.12]{FK80}.
Thus $\tilde q(X) \ne 0$. 
\end{proof}

\subsection{The non-isotrivial case} \label{subsectionnonisotrivial}

In the situation of Setup \ref{setup}, assume furthermore that
the elliptic fibration $f$ is not isotrivial. Let $F$ be a general fibre, then its Kodaira-Spencer class is not zero. Thus we have a non-split extension defined by the restriction of \eqref{sequencemain} to $F$ 
$$
0 \rightarrow \sO_F \rightarrow \Omega_X|_F \rightarrow \Omega_F \simeq \sO_F \rightarrow 0.
$$
Denote by $\zeta_F \rightarrow \PP(\Omega_X|_F)$ the tautological class. Then the (1,1)-class $\zeta_F$ is nef with $\zeta_F^2=0$, moreover it is represented by the current of integration over the curve $C$ defined by the quotient $\Omega_X|_F \rightarrow \sO_F$. We recall the following result:
\begin{lemma} \cite[Ex. 1.7]{DPS94} \label{lemmaDPS}
Let $h_F$ be any singular metric on $\zeta_F$ such that the curvature current $\Theta_{h_F}(\zeta_F)$ is positive. Then 
$$
\Theta_{h_F}(\zeta_F) = [C]
$$
where $[C]$ is the current of integration over $C$.
\end{lemma} 

\begin{proposition} \label{propositionLpseff}
In the situation of Setup \ref{setup}, assume that 
$f$ is not isotrivial. If $\zeta$ is pseudoeffective, the line bundle $f^*\Omega_ B (D)$
is pseudoeffective.
\end{proposition}

\begin{proof}
By assumption there exists a singular metric $h$ on $\zeta$ such that $\Theta_{h}(\zeta) \geq 0$. Then we can consider the Siu decomposition
$$
\Theta_{h}(\zeta) = \sum_k \nu(\Theta_{h}(\zeta), Y_k) [Y_k] + P
$$
where $Y_k \subset \PP(\Omega_X)$ are prime divisors and $P$ is a positive closed current such that the (countably many) irreducible components of $E_+(P)$ \cite[Sect.2.2.1]{Bou04} have codimension at least two. 
Let now $F$ be a very general $f$-fibre, so that $\PP(\Omega_X|_F)$ does not contain any positive dimensional irreducible component
of $E_+(P)$ 
then the restriction $h_F$ of the metric $h$ to $\PP(\Omega_X|_F)$ is well-defined and yields a singular metric on $\zeta_F$.
Moreover we have a decomposition
$$
\Theta_{h_F}(\zeta_F) = (\Theta_{h}(\zeta))|_F =
\sum \nu(\Theta_{h}(\zeta), Y_k) [Y_k \cap F] + P|_F.
$$
By Lemma \ref{lemmaDPS} we see that $P|_F=0$, and there exists a unique $Y_k$ (say $Y_1$) such that $Y_1 \cap F=C$
and $\nu(\Theta_{h}(\zeta), Y_1)=1$. 

Since the intersection $Y_1 \cap F$ coincides with $C$ on all the general fibres and the quotient
$\Omega_X|_F \rightarrow \Omega_F$ is the restriction of the quotient $\Omega^1_X \to I_Z \otimes \omega_{X/B}(-D)$
we obtain that $Y_1=Y$. Thus we have
$$
\Theta_{h}(\zeta) -[Y] = \sum_{k \geq 2} \nu(\Theta_{h}(\zeta), Y_k) [Y_k] + P
$$
which is a positive current. By \eqref{formulaY} this implies the statement.
\end{proof}

By Proposition \ref{propA} and \ref{propcover2}, we conclude

\begin{corollary} \label{cor:noniso} Let $X$ be a smooth projective surface, and
let $f: X \to B$ be a non-isotrivial elliptic fibration 
Then the following are equivalent:
\begin{enumerate}
\item $\Omega^{1}_X$ is  pseudoeffective ;
\item we have $\widetilde q(X) > 0$ ;
\item we have $H^0(X,S^{m}\Omega^1_X)) \ne 0$ for some positive integer $m$
\end{enumerate}
\end{corollary} 

In summary, Conjecture \ref{conj:nv} holds for non-isotrivial elliptic fibrations.

\section{Elliptic surfaces: the isotrivial case}
\label{sectionisotrivial}

In this section we treat isotrivial elliptic fibrations and prove parts (b) and (c) of Theorem \ref{theoremmain}.

\subsection{Notation} \label{subsectionisotrivialnotation}

In the situation of Setup \ref{setup}, 
assume that
$$
f : X \rightarrow B
$$
is relatively minimal and isotrivial. Denote by $E$ the elliptic curve such that a general $f$-fibre is  isomorphic to $E$.
By \cite[Sect.2]{Ser96} there exists a smooth curve $C$ and a finite group $G$ 
acting diagonally on the product $C \times E$ such that $X$ is birational 
to the quotient $(C \times E)/G$ and the fibration $f$ corresponds to the fibration
$(C \times E)/G \rightarrow C/G \simeq B$ induced by projection on the first factor.

More precisely, denote by $\holom{q}{C \times E}{(C \times E)/G}$ the quotient map, and 
by $\holom{\bar p_C}{(C \times E)/G}{C/G}$ 
the map induced by the projection $p_C$. 
Denote by $\holom{\lambda}{R}{(C \times E)/G}$ the minimal resolution
of singularities, then the exceptional divisors are Hirzebruch-Jung strings \cite[2.0.2]{Ser96} and the singular fibres of 
$$
f_R:= \bar p_C \circ \lambda : R \rightarrow B
$$ 
are described in \cite[Thm.2.1]{Ser96}. Following Serrano, we call $f_R$ the standard model of the isotrivial fibration $f$.

The standard isotrivial fibration $f_R$ factors through its relative minimal
model. Since we assumed that $f$ is relatively minimal and the relative minimal
model of an elliptic surface is unique, we have a birational morphism 
$\holom{\mu}{R}{X}$ such that $f_R = f \circ \mu$. We summarise the construction in a commutative diagram:
\begin{equation} \label{diagramconstruction}
\xymatrix{
R \ar[rd]^{\lambda}  \ar[d]_{\mu}  & C \times E \ar[d]^q 
\\
X \ar[d]_{f}  & (C \times E)/G \ar[d]^{\bar p_C} 
\\
B \ar@{=}[r] & C/G
}
\end{equation}

In general the birational map $\mu$ is not an isomorphism  \cite[(2.4)]{Ser96},
but as shown in Lemma \ref{lem:char},  it is an isomorphism unless $E$ has complex multiplication.

\subsection{Relatively minimal standard isotrivial fibrations}

\begin{assumption} \label{assumptionstandard}
In the {\em whole subsection} we assume that 
$$
f : X \rightarrow B
$$
is a relatively minimal, isotrivial elliptic fibration that is {\em standard},
i.e. the birational map $\mu$ (cf. Notation \ref{subsectionisotrivialnotation}) is an isomorphism. 
For simplicity of the exposition we thus identify $X=R$.
\end{assumption}

The following lemmas are well-known, we include them for lack of reference:

\begin{lemma} \label{lem:fibers}
Using the Notation \ref{subsectionisotrivialnotation}, let 
$x \in C$ be a point with non-trivial stabiliser
$G_x$ on $C$.  then $G_x$ acts either by translation or as the involution
$z \mapsto -z$ on the elliptic curve $x \times E$.

In particular all the singular fibers of $f$ are either multiple elliptic curves 
or curves of type $I_0^*$ in Kodaira's classification, see e.g.,  \cite[V, Sect. 7, Table 3]{BHPV04}. 
\end{lemma} 

\begin{proof}  We follow 
the description of the singularities of $(C \times E)/G$ in
\cite[2.0.2]{Ser96}: the group $G$ acts on $C \times E$,
let $b \in B$ a point with non-trivial stabiliser $G_b \subset G$.
The group $G_b$ is cyclic \cite[III.7.7,Cor.]{FK80}. If it acts freely  on $E$, the quotient
$(B' \times E)/G$ is smooth near the fibre which is multiple elliptic. 
Assume now that there exists a point $e \in E$ with non-trivial stabiliser
$G_{b,e} \subset G_b$. Then $G_{b,e} \subset Aut(E, e)$. In view of \cite[IV, Cor.4.7]{Har77} this strongly limits the possibilities:

If $G_{b,e} = \Z_2$, the group acts as $z \mapsto -z$ on $E$ and we are done. 

 If $G_{b,e} = \Z_4$, the group acts as $z \mapsto i z$ on the elliptic curve
 $\C/\Z \oplus i \Z$. This action fixes the origin and $\frac{1+i}{2}$, so
 we obtain two singularities of type $A_{1,4}$ in the quotient.
 The points $\frac{1}{2}$ and $\frac{i}{2}$ have non-trivial stabiliser, but are in the same orbit, so we obtain one singular point of type $A_{1,2}$ in the quotient.
The minimal resolution of  $R \rightarrow (C \times E)/G$ 
thus yields a fibre with a central component of multiplicity $4$, two components
of multiplicity $1$ and self-intersection $-4$ and one component 
of multiplicity $2$ and self-intersection $-2$. This
fibre is the log-resolution of a fibre of Kodaira's type $III$, but it is not relatively minimal.

If finally $G_{b,e} = \Z_6$, then we see in a similar fashion that the action of $\zeta$ on the elliptic curve
 $\C/\Z \oplus \zeta \Z$ (with $\zeta=e^{i \frac{2\pi}{3}}$) leads to a log resolution of the configuration 
 obtained for type $IV$ and the action of $-\zeta$ leads to a log resolution of the configuration 
 obtained for type $II$. The contradiction is then as before.

\end{proof} 

\begin{lemma} \label{lem:char} 
Let $f: X \to B$ be a minimal isotrivial elliptic fibration as in Notation \ref{subsectionisotrivialnotation}.
Assume that if $x \in C$ is a point with non-trivial stabiliser
$G_x$ on $C$,  then $G_x$ acts either by translation or as the involution
$z \mapsto -z$ on the elliptic curve $x \times E$. Then $f: X \to B$ is standard. 
In particular, if $E$ does not have complex multiplication, then $f: X \to B$ is standard. 

\end{lemma} 

\begin{proof}  Under our assumptions, the singular fibers of $\overline{p}_C \circ \lambda = f \circ \mu$ are either multiple elliptic curves or of type $I_0^*$. 
But then $\mu$ must be an isomorphism, due to Kodaira's classification, applied to $f$. 
\end{proof} 


We have the following crucial

\begin{lemma} \label{lemmanumberZ}
Using the Notation \ref{subsectionisotrivialnotation},
let $Z$ denote the set of points in $x \in C$ such that $G_x$ 
acts as the involution $z \mapsto -z$ on the elliptic curve $x \times E$.

If $f^* \Omega_B(D)$ is not pseudoeffective, then $Z$ has at least $2g(C)-1$ elements.
\end{lemma}

\begin{proof}
Since $f^* \Omega_B(D)$ is not pseudoeffective, we have $B \simeq \PP^1$.

The action of $G$ on the curve $C$ defines a Galois cover $\holom{\psi}{C}{C/G \simeq \PP^1}$ of degree $d=|G|$. By the Hurwitz formula we have
$$
2 g(C)-2 = \deg K_C = - 2 d + \deg R,
$$
where $R$ is the ramification divisor of $\psi$. If $x \in Z$, then it is a point with stabiliser $G_x \simeq \Z_2$, hence $\psi$
ramifies with order $2$ in $x$. Thus $\deg R_x=1$ and
$$
\# (Z) = 2 g(C)-2 + 2d - \deg R_t,
$$ 
where $R_t \subset R$ is the ramification corresponding to points $y \in C$ such that $G_y \simeq \Z_{m_y}$ acts by translation on $y \times E$ (see Lemma \ref{lem:fibers}). If $y \in C$ is a point with this property, every point in its orbit $G.y$ has the same property. Since the orbit has length
$\frac{d}{m_y}$, this induces a ramification divisor of order $d \frac{m_y-1}{m_y}$. Now note that the corresponding $f$- fibre over $\bar y=\psi(y)$ is multiple
elliptic with multiplicity $m_y$, so it defines a component of $D$ that is numerically equivalent to $\frac{m_y-1}{m_y} F$ where $F$ is a general fibre of $f$. Thus we see that $f^* \Omega_{\PP^1}(D)$ is pseudoeffective if 
$$
\sum_{\bar y \in \PP^1, G_y \mbox{\small acts by translation}}  \frac{m_y-1}{m_y}  \geq 2.
$$
Suppose now that this is not the case: then we have 
$$
\deg R_t = \sum_{y \in C, G_y \mbox{\small acts by translation}}  d  \frac{m_y-1}{m_y} < 2d,
$$
hence $\# (Z) = 2 g(C)-2 + 2d - \deg R_t>2g(C)-2$.
\end{proof}

\begin{construction}  \label{construction} 

We again use the Notation \ref{subsectionisotrivialnotation}.
Let $\bar A$ be an ample Cartier divisor on $(C \times E)/G$, 
and set $A_X = \lambda^* \bar A$.
Then by
definition (and \cite[Lemma 2.2]{HLS20}) the vector bundle 
$\Omega_X$ is not pseudoeffective if and only if there exists 
a $c>0$ such that for all $i,j \in \N$ such that $i > cj$ one has
$$
H^0(X, S^{i} \Omega_X \otimes \sO_X(j A_X)) = 0.
$$
Denote by $D \subset X$ the exceptional locus of $\lambda$, then we have morphisms
$$
\xymatrix @R=0.5cm {
H^0(X, S^{i} \Omega_X \otimes \sO_X(j A_X)) 
\ar @{^{(}->}[d]
\\
H^0(X \setminus D, S^{i} \Omega_X \otimes \sO_X(j A_X)) 
\ar[d]^{\simeq}
\\
H^0((C \times E)/G \setminus \lambda(D), S^{i} \Omega_{(C \times E)/G} \otimes \sO_{(C \times E)/G}(j \bar A)) 
\ar @{^{(}->}[d]
\\
H^0(C \times E \setminus \fibre{q}{\lambda(D)}, S^{i} \Omega_{C \times E} \otimes 
\sO_{C \times E}(j q^* \bar A)) 
\ar[d]^{\simeq}
\\
H^0(C \times E, S^{i} \Omega_{C \times E} \otimes 
\sO_{C \times E}(j q^* \bar A)) 
}
$$
where the last isomorphism is due to the fact that $S^i \Omega_{C \times E}$
is reflexive, since $C \times E$ is smooth. 

Finally let $A_C$ be an ample divisor on $C$ and $A_E$ 
an ample divisor of degree one on $E$. Set
$$ 
A := p_C^*(A_C) \otimes p_E^*(A_E) 
$$ 
with the canonical projections $p_C: C \times E \to C$ and $p_E: C \times E \to E$. 
 Then for some $l \in \N$ sufficiently high we have an inclusion 
 $$\sO_{C \times E}(q^* \bar A) \hookrightarrow 
\sO_{C \times E}(l A),$$ 
so for all $i,j \in \N$ we obtain an inclusion
\begin{equation} \label{includesymmetric}
\Phi = \Phi_{i,j,l}:  H^0(X, S^{i} \Omega_X \otimes \sO_X(j A_X)) 
\hookrightarrow
H^0(C \times E, S^{i} \Omega_{C \times E} \otimes 
\sO_{C \times E}(jl A)). 
\end{equation}

\end{construction}

\begin{definition} \label{definitioninduce}
We say that a $\eta \in H^0(C \times E, S^{i} \Omega_{C \times E} \otimes 
\sO_{C \times E}(jA))$ induces a holomorphic symmetric differential 
with values in $A_X$ on $X$ if it is in the image of  an inclusion $\Phi$ in  \eqref{includesymmetric}.
\end{definition}

\begin{remarks*}
The definition is a slight abuse of terminology, since the inclusion $\Phi$ in  \eqref{includesymmetric} is only defined for $j$ divisible by $l$. Since in the definition of pseudoeffectivity we can always replace $i$ and $j$ by $il$ and $jl$,
we will, for the simplicity of notation, ignore this point.

Note also that the chain of inclusions does not use that $C$ is proper, so the terminology also applies to an analytic open subset $\Delta \times E \subset C \times E$ with some subgroup $G_x \subset G$ acting on $\Delta \times E$. 
\end{remarks*}

\subsection{The nonvanishing conjecture for standard isotrivial fibrations} 

The goal of this subsection is to prove the following

\begin{theorem} \label{theoremisotrivialstandard}
Let $X$ be a smooth projective surface that admits 
relatively minimal, isotrivial elliptic fibration
$$
f : X \rightarrow B
$$
that is {\em standard} (see Subsection \ref{subsectionisotrivialnotation}).
If $f^* \Omega_B(D)$ is not pseudoeffective, then $\Omega_X$ is not pseudoeffective.
\end{theorem}

By Proposition \ref{propA} and \ref{propcover2} we thus obtain:

\begin{corollary} \label{cor:isotrivialstandard} 
Let $f: X \to B$ be a standard isotrivial fibration; e.g., $f$ is an isotrivial elliptic fibration such that the general fiber does not have complex multiplication. 
Then the following are equivalent:
\begin{enumerate}
\item $\Omega^{1}_X$ is  pseudoeffective ;
\item we have $\widetilde q(X) > 0$ ;
\item we have $H^0(X,S^{m}\Omega^1_X)) \ne 0$ for some positive integer $m$
\end{enumerate}
\end{corollary} 

The proof of Theorem  \ref{theoremisotrivialstandard} is done by showing that if $i \gg j$, a non-zero 
$$
\eta \in H^0(C \times E, S^{i} \Omega_{C \times E} \otimes 
\sO_{C \times E}(jA))
$$ 
does not induce a holomorphic symmetric differential on $X$. In other words, $\Phi = 0$. Since the details are somewhat technical, let us first recall and reprove a result of Sakai.

\begin{example} \cite[\S 4, (D)]{Sak79}
Let $C$ be a hyperelliptic curve, and let $\tau = (i_C, i_E)$ be the involution
on $C \times E$ defined by the hyperelliptic involution $i_C$ on $C$ 
and the map $i_E : E \rightarrow E, z \mapsto -z$ on an elliptic curve $E$. 
The minimal resolution $X \rightarrow (C \times E)/\langle \tau \rangle$,
has an isotrivial fibration $f: X \rightarrow \PP^1$ such that all the singular fibres are of type $I_0^*$, in particular it is relatively minimal and standard.
Then one has
$$
H^0(X, S^i \Omega_X)=0 \qquad \forall i \in \N.
$$
{\em Proof:} Let 
$$
p_C^* \alpha \otimes p_E^* \beta \in H^0(C \times E, S^i \Omega_{C \times E})
= \bigoplus_{l=0}^i H^0(C \times E, p_C^* \omega_C^{\otimes l} \otimes p_E^* \omega_E^{\otimes i-l})
$$
be a rank one tensor, i.e. $\alpha \in H^0(C, \omega_C^{\otimes l})$
and $\beta \in H^0(E, \omega_E^{\otimes i-l})$ for some $l \in {0, \ldots, i}$.

Assume that $\alpha \otimes \beta$ induces a holomorphic symmetric differential
on $X$. Arguing by contradiction we assume that $\alpha \neq 0$ and $\beta \neq 0$.

If $(x,0)$ is a fixed point of $\tau$, choose local coordinates
$z_1$ on $C$ and $z_2$ on $E$ such that locally near $(x,0)$, the 
involution is given by $(z_1, z_2) \mapsto (-z_1, -z_2)$.
In these local coordinates we write $\alpha = f_\alpha(z_1) d z_1^l$
and $\beta = f_\beta(z_2) d z_2^{i-l}$. For a general point in the exceptional divisor over the point $\overline{(x,0)}$, we can choose local coordinates $(u, v)$ on $X$ such that
$z_1=u^2, z_2 = uv$. In these coordinates the exceptional divisor is given by $u=0$. Substituting $(z_1, z_2)$ by these coordinates we see that
$\alpha \otimes \beta$ induces the meromorphic symmetric differential
$$
f_\alpha(\sqrt{u}) \cdot  f_\beta(\sqrt{u}) \cdot
du^l \frac{(u dv+ v d u)^{i-l}}{(2 \sqrt{u})^i}.
$$
on $X$. 
Looking at the term of $du^i$ we see that the differential is holomorphic along the exceptional divisor if and only if 
$$
\frac{f_\alpha(\sqrt{u}) \cdot  f_\beta(\sqrt{u})}{(\sqrt{u})^i}
$$
is holomorphic, i.e. if and only if $\alpha \otimes \beta$ vanishes with
order at least $i$ in the fixed point. Since $0 \neq \beta \in  H^0(E, \omega_E^{\otimes i-l}) \simeq
H^0(E, \sO_E)$ does not vanish, this shows that $\alpha \in H^0(C, \omega_C^{\otimes l})$
vanishes with order $i$ in $x$. 
Since the involution $i_C$ has $2g(C)+2$ fixed points, we obtain
that $\alpha$ vanishes along a divisor of degree at least $i \cdot (2g(C)+2)$.
Since $\deg \omega_C^{\otimes l} = l \cdot (2g(C)-2) < i \cdot (2g(C)+2)$
we obtain that $\alpha = 0$.
Since $H^0(E, \omega_E^{\otimes i-l})$ has dimension one for every $i-l$ one can reduce the general case to rank one tensors, so the statement follows.
 This settles the proof of the example. 
\end{example}

In the proof of Theorem \ref{theoremisotrivialstandard} we have
$h^0(E, \omega_E^{\otimes i-l} \otimes \sO_E(j A_E))>1$, so the symmetric differentials are not global rank one tensors. Somewhat surprisingly, this leads to a much weaker local obstruction (cf. \cite[Prop.3.2]{BTV19}), i.e., the vanishing order 
of the symmetric differential in a fixed point can be strictly smaller than $i$.
We will improve this local estimate by taking into account that the vanishing
order along $E$ is bounded by $j \deg A_E$.

\begin{the local obstruction - setup}  {\rm 
Let us describe the local obstruction for a holomorphic symmetric differential on $C \times E$ to induce a holomorphic symmetric differential on $X$:
using the notation of Lemma \ref{lemmanumberZ}, fix a point $x \in Z \subset C$,
and let $\tau_x$ be the generator of $G_x \simeq \Z_2$. 
Let $0 \in E$ be a fixed point of $\tau_x$, for this local computation we choose $A_E:=0$ to be the corresponding
ample divisor of degree one\footnote{We make this choice in order to simplify the notation. Since the obstruction depends only on a small neighbourhood of $0 \in E$, the statement in Corollary \ref{corollarylocalsymmetric} is independent of this choice.}.

Let $x \in \Delta \subset C$ be a small disc and choose a local coordinate $z_1$ such that the action
of $G_x$ is given by $z_1 \mapsto -z_1$ (in particular $x=0$). We have
$$
H^0(\Delta \times E, \sO_{\Delta \times E}(j A)) \simeq \C\{z_1\} \otimes H^0(E, \sO_E(j A_E)),
$$
where $z_1$ is a local coordinate on $\Delta$ and $\C\{z_1\}$ denotes the algebra of convergent power series in $z_1$. 

Since $E$ is an elliptic curve, we have $h^0(E, \sO_E(j A_E))= j$, and we may choose a basis of sections $s_{j,0}, \ldots, s_{j, j-2}, s_{j,j}$
such that $s_{j,k}$ vanishes with order exactly $k$ in the neutral element $0$ of the elliptic curve. In particular we have a $\C$-basis of $H^0(\Delta \times E, \sO_{\Delta \times E}(j A))$
$$
z_1^{n-k} s_{j, k}  \qquad n \in \N, \ k=0, \ldots, j-2, j.
$$
Note that by construction $z_1^{n-k} s_{j, k}$ vanishes with order exactly $n$ in $(0,0)$.

Since $E$ is an elliptic curve, its cotangent bundle is trivial and we denote by $d z_2$ a global generator of $\Omega_E$. Let now
$$
H^0(\Delta \times E, S^i \Omega_{\Delta \times E} \otimes \sO_{\Delta \times E}(j A))
\simeq
\oplus_{l=0}^i H^0(\Delta \times E, \sO_{\Delta \times E}(j A)) \otimes dz_1^l dz_2^{i-l}
$$
be the space of symmetric differentials with values in $\sO_{\Delta \times E}(j A)$. Then we can decompose
$$
H^0(\Delta \times E, S^i \Omega_{\Delta \times E} \otimes \sO_{\Delta \times E}(j A)) = \oplus_{n \in \N} V_n
$$
where $V_n$ is generated by 
$$
z_1^{n-k} s_{j, k} dz_1^l dz_2^{i-l}  \qquad k=0, \ldots, j-2, j, \qquad l=0 \ldots i.
$$
Let now $\omega \in H^0(\Delta \times E, S^i \Omega_{\Delta \times E} \otimes \sO_{\Delta \times E}(j A))$, and let
$$
\omega = \sum_{n \in \N} \omega_n
$$
be the decomposition such that $\omega_n \in V_n$. 
Finally let 
$$
M := z_1 d z_2 - s_{1,1} d z_1 
$$
be the holomorphic $1$-form with values in $A$ giving in local coordinates the form $z_1 d z_2 - z_2 d z_1$. }

\end{the local obstruction -  setup} 

In the setup just introduced, the following holds :

\begin{lemma}
Let $S$ be the minimal resolution
of $(\Delta \times E)/G_x$. If $\omega$ induces a holomorphic symmetric differential with values in $A_{S}$ on $S$ (see Definition \ref{definitioninduce}),
then for all $n \in \N$ the form $\omega_n$ induces a holomorphic symmetric differential with values in $A_{S}$. Moreover, if $n < i$, there exists a differential 
$$
\eta_n \in H^0(\Delta \times E, S^{\frac{i+n}{2}} \Omega_{\Delta \times E} \otimes 
\sO_{\Delta \times E}((j-\frac{i-n}{2}) A))
$$
such that $\omega_n = \eta_n \times M^{\frac{i-n}{2}}$.
\end{lemma}

\begin{remark*}
Since $G_x \simeq \Z_2$ acts locally as $(z_1, z_2) \mapsto (-z_1, -z_2)$,
one has $n = i \mod 2$ \cite[Sect.3]{BTV19}, so
$\frac{i+n}{2}$ and $\frac{i-n}{2}$ are positive integers.
\end{remark*}

\begin{proof}
The property of being holomorphic is local, so we can just apply the proof of \cite[Prop.3.2]{BTV19}.
\end{proof}

Since $H^0(E, \sO_E((j-\frac{i-n}{2}) A_E))=0$ for $j-\frac{i-n}{2}<0$ this implies :

\begin{corollary} \label{corollarylocalsymmetric}
If $\omega \in H^0(\Delta \times E, S^i \Omega_{\Delta \times E} \otimes \sO_{\Delta \times E}(j A))$ 
induces a holomorphic symmetric differential with values in $A_{S}$ on $S$,
then 
$$
\omega \in H^0(\Delta \times E, I_{0,0}^n \otimes S^i \Omega_{\Delta \times E} \otimes \sO_{\Delta \times E}(j A))
$$
where $I_{0,0}$ is the ideal sheaf of the origin and $n \geq i - 2j$. 
\end{corollary}

\begin{proof}[Proof of Theorem \ref{theoremisotrivialstandard}]
We use the notation of Subsection \ref{subsectionisotrivialnotation} and Construction \ref{construction}.
Denote by $Z \subset C$ the set of points in $x \in C$ such that the stabiliser 
$G_x = \langle \tau_x \rangle$ 
acts as the involution $z \mapsto -z$ on the elliptic curve $x \times E$.
Since $f$ is standard and $f^* \Omega_B(D)$ is not pseudoeffective, we know by Lemma \ref{lemmanumberZ}
that $Z$ has at least $2g-1$ element where $g=g(C)$.

For every $x \in Z$ we fix a point $x' \in (x \times E)$ such that $(x, x')$ is a fixed point of $\tau_x$.

Fix some rational $\epsilon>0$ such that  $\frac{2g-2}{2g-1}+\epsilon<1$ and 
fix a positive integer $N$ such that 
$$ 
(\frac{2g-2}{2g-1} + \epsilon) N \in \mathbb N.
$$ 

Assume now that 
$$
\eta \in H^0(C \times E, S^{i} \Omega_{C \times E} \otimes \sO_{C \times E}(j A))
$$
induces a holomorphic symmetric differential with values in $A_X$ on $X$ (see Definition \ref{definitioninduce}). 
Then for every $x \in Z$, the restriction to some neighbourhood $\Delta \times E$
induces a holomorphic symmetric differential on the minimal resolution $S$
of $(\Delta \times E)/G_x$. Thus by 
Corollary \ref{corollarylocalsymmetric} we have for every $i \in \N$ that
$$
\eta \in 
H^0(C \times E,
(\otimes_{x \in Z} I_{(x,x')}^n) \otimes S^i \Omega_{C \times E} \otimes \sO_{\Delta \times E}(j A).
)
$$
with $n \geq i - 2j$.

Observe that if $i \geq \frac{2}{1 - \frac{2g-2}{2g-1}-\epsilon} j$   one has
$$
n \geq (\frac{2g-2}{2g-1}+\epsilon) i. 
$$
Thus
we get an inclusion
$$
\eta^N \in 
H^0(C \times E,
(\otimes_{x \in Z} I_{(x,x')}^{(\frac{2g-2}{2g-1}+\epsilon) N i}) \otimes S^{Ni }\Omega_{C \times E} \otimes \sO_{C \times E}(N j A)
).
$$
We claim that this last vector space is zero
for $i \geq \frac{2}{1 - \frac{2g-2}{2g-1}-\epsilon} j$. 
As explained after Theorem \ref{theoremisotrivialstandard}, this shows the statement.

{\em Proof of the claim.}
Recall that
$$
S^i \Omega_{C \times E} = \oplus_{l=0}^i p_C^* \omega_C^l \otimes p_E^*
\omega_E^{i-l}
\simeq \oplus_{l=0}^i p_C^* \omega_C^l
$$
since $\omega_E \simeq \sO_E$. 
Since $\omega_C^l \hookrightarrow \omega_C^i$ for $l<i$ we are reduced to showing that 
$$
H^0(C \times E,
(\otimes_{x \in Z} I_{(x,x')}^{(\frac{2g-2}{2g-1}+\epsilon)Ni}) \otimes p_C^* \omega_C^{Ni} \otimes \sO_{C \times E}(N j A)
)
= 0
$$
for $i \gg j$. 
By Definition \ref{definitionpseff} this is the same as proving that
$$
\otimes_{x \in Z} I_{(x,x')}^{(\frac{2g-2}{2g-1}+\epsilon) N} \otimes p_C^* \omega_C^N
$$
is not strongly pseudoeffective. Arguing by contradiction assume that the sheaf is strongly pseudoeffective.

Let $\holom{\mu}{S}{C \times E}$ be the blowup
of the (reduced) set $\cup_{x \in Z} (x,x') \subset C \times E$
and denote by $\sO_S(1)$ the tautological sheaf, i.e. the ideal sheaf
of the exceptional divisor of $\mu$. 
Recall that $S$ is isomorphic to the blowup of the ideal sheaf  
$\otimes_{x \in Z} I_{(x,x')}^{(\frac{2g-2}{2g-1}+\epsilon) N}$ (e.g. \cite[II,Ex.7.11]{Har77}),
so by Corollary \ref{corollarypsefftautological} the line bundle
$$
\sO_S((\frac{2g-2}{2g-1}+\epsilon) N) \otimes \mu^* p_C^* \omega_C^N
$$
is pseudoeffective. By Lemma \ref{lemS1}, applied to the elliptic fibration $p_C \circ \mu: S \rightarrow C$, there exists $m \in \N$ and a numerically trivial line bundle $M$ on $C$ such that
$$
H^0(S, \sO_S((\frac{2g-2}{2g-1}+\epsilon) mN) \otimes \mu^* p_C^* (\omega_C^{mN} \otimes M^*)) \neq 0.
$$
Thus we have
\begin{equation} \label{nonvanishing}
H^0(C \times E, \otimes_{x \in Z} I_{(x,x')}^{(\frac{2g-2}{2g-1}+\epsilon) mN}
\otimes \mu^* p_C^* (\omega_C^{mN} \otimes M^*)) \neq 0.
\end{equation}

Yet by Lemma \ref{lemmanumberZ} we know that $Z$ has at least $2g-1$ elements,
hence 
$$ 
\deg \bigl( \otimes_{x \in Z} I_{x}^{(\frac{2g-2}{2g-1}+\epsilon) mN} \bigr) \geq mN ( \frac{2g-2}{2g-1} + \epsilon) \cdot 2g-1 > mN(2g-2). 
$$
Since $\deg (\omega_C^{mN} \otimes M^*) = mN(2g-2)$, we obtain 
a contradiction to \eqref{nonvanishing}. 
\end{proof}

\subsection{Isotrivial fibrations and the Zariski decomposition}
\label{subsectionisotrivialzariski}

Let $\holom{f}{X}{\PP^1}$ be an isotrivial elliptic fibration over a rational curve, and assume that $\Omega_X$ is pseudoeffective. 
Since the proof of Theorem \ref{theoremisotrivialstandard} is a bit tedious, we present
here a more conceptual approach based on the ideas of Subsection \ref{subsectionnonisotrivial}. The considerations of this section are independent of whether $f$ is standard or not. 

We use the notations of the setup \ref{setup}. 
Let $\zeta$ be the tautological class of $\PP(\Omega_X)$.
If the elliptic fibration $f$ is not almost smooth, we would like to show
that the subvariety 
$$
Y := \mathbb P(\sI_Z \otimes \omega_{X/B}(-D)) \subset \PP(\Omega_X)
$$ defined by
$f$ is in the negative part of the Zariski decomposition of $\zeta$. If $f$ is isotrivial, the restriction $\PP(\Omega_X|_F)$ over a general fibre $F$ is 
isomorphic to $\PP^1 \times F$, so the proof of Proposition \ref{propositionLpseff} 
does not apply. We therefore have to use some global information
to explicitly compute the restriction $\zeta|_Y$.

\begin{proposition} \label{prop:isotrivial}
In the situation of Setup \ref{setup}, assume that $f$ is isotrivial, relatively minimal and not almost smooth. 
Then $\zeta|_Y$ is not pseudoeffective.
\end{proposition}

\begin{proof} By Corollary \ref{corollarypsefftautological} the statement is equivalent to showing that $I_Z \otimes \omega_{X/B}(-D)$ is not strongly pseudoeffective.
By Corollary \ref{corpseffkappa} this is equivalent to showing that
$$
H^0(X, I_Z^k \otimes (\omega_{X/B}(-D))^{\otimes k}) = 0
$$
for all $k \in \N$.  Recall the canonical bundle formula \eqref{cbf}, formula  \eqref{formulaKXL} 
and \cite[III,Prop.11.4, Rem.11.5]{BHPV04}:
$$
\chi(X, \sO_X) = \frac{e(X)}{12} = \frac{1}{12} \sum_{b \in K} e(X_b).
$$
Here the set $K \subset B$ consists of all the points such that the reduction of the scheme-theoretic fibre $X_b$ is not an elliptic curve.
From these facts, we obtain, 
$$
K_{X/B}-D \sim_\Q 
 \sum_{b \in K} (\frac{e(X_b)}{12} F - D_{0,b}).
$$
Now the proof is finished by observing that 
$$
\kappa(X,  \sI_{Z,b} \otimes \frac{e(X_b)}{12} F) = -\infty
$$
for all $b \in K$. This is done by a tedious case by case calculation, 
using that, by the proof of Lemma \ref{lem:fibers}, the non-multiple fibres
are of Kodaira's type $II, III, IV$ or $I_0^*$.
The details are left to the interested reader.
\end{proof} 

As an application, we obtain

\begin{corollary} \label{cor:Zar} 
Let $X$ be a smooth projective surface such that $K_X$ is nef and $\kappa(X)=1$. 
Denote by $\zeta$ the tautological divisor on $\holom{\pi}{\PP(\Omega_X)}{X}$. 
If $\zeta$ is pseudoeffective and nef in codimension one, then
we have $\widetilde q(X) > 0$. In particular, by \cite[Prop.2.2]{Ane18},
we have $H^0(X,S^{m}\Omega^1_X) \ne 0$ for some positive integer $m$.
\end{corollary}

\begin{proof}
If the Iitaka fibration is isotrivial and not almost smooth, we see by Proposition \ref{prop:isotrivial} that the restriction $\zeta|_Y$ is not pseudoeffective.
Thus $Y$ is in the negative part of the divisorial Zariski decomposition, hence $\zeta$ is not nef in codimension one.

Thus we know that $f$ is almost smooth or not isotrivial.
If the Iitaka fibration is almost smooth, then $c_2(X) = 0$ \cite[III,Prop.11.4]{BHPV04} and thus $q(X) > 0$ as already noticed at the beginning of Section \ref{sectionelliptic}. 

If the Iitaka fibration $\holom{\varphi}{X}{B}$ is not isotrivial, we conclude by Corollary
\ref{cor:noniso}. 
\end{proof}

\begin{remark*}
By Proposition \ref{prop:isotrivial} the restriction $\zeta|_Y$ is not pseudoeffective,
so by divisorial Zariski decomposition there exists a $c>0$ such that $\zeta-cY$ is pseudoeffective. If we prove that $\zeta-cY$ is pseudoeffective for some $c \geq 1$,
we obtain the nonvanishing conjecture as in Corollary \ref{cor:noniso}. 
However this is not obvious, even in simple situations.
\end{remark*}

\begin{appendix}

\section{Fundamental group of elliptic surfaces} \label{appendixfinite}

The following statement is essentially a consequence of \cite[Sect. 3.5]{Cam04}

\begin{lemma}
Let $\holom{f}{X}{B}$ be a smooth elliptic surface such that the cotangent bundle
$\Omega_X$ is not pseudoeffective. Then the fundamental group of $X$ is finite.
\end{lemma}

\begin{proof}
By Proposition \ref{prop:bir} we can assume without loss of generality that $X$
is minimal. Since $\Omega_X$ is not pseudoeffective, we have $B \simeq \PP^1$.
If $f$ is almost smooth, there exists an \'etale cover $B' \times E \rightarrow X$
with $E$ an elliptic curve (see introduction to Section \ref{sectionelliptic}). 
In particular $\Omega_{B' \times E}$ is pseudoeffective, a contradiction to Proposition \ref{prop:cover}.
Thus $f$ is not almost smooth, hence by \cite[Lemma 1.39]{CZ79} and \cite[Cor.12.10]{Cam11} one has $\pi_1(X) \simeq \pi_1(B, \Delta)$, where $(B, \Delta)$ 
is the orbifold structure defined by the multiple fibres. Since $B$ is a curve, by \cite[App.C]{Cam98}, we can find a finite \'etale cover $X' \rightarrow X$ 
such that the orbifold divisor of $X' \rightarrow B'$ is empty. Since $\Omega_{X'}$
is not pseudoeffective by Proposition \ref{prop:cover}, we have $B' \simeq \PP^1$.
Thus $\pi_1(X') \simeq \pi_1(B') \simeq \{ 1 \}$.
\end{proof}

\end{appendix}


\begin{thebibliography}{BHPVdV04}

\bibitem[Ane18]{Ane18}
Fabrizio Anella.
\newblock Rational curves on genus one fibrations.
\newblock {\em arXiv preprint}, 1809.09485, 2018.

\bibitem[AT82]{AT82}
Vincenzo Ancona and Giuseppe Tomassini.
\newblock {\em Modifications analytiques}, volume 943 of {\em Lecture Notes in
  Mathematics}.
\newblock Springer-Verlag, Berlin, 1982.

\bibitem[AW97]{AW95}
Marco Andreatta and Jaros\l aw~A. Wi\'{s}niewski.
\newblock A view on contractions of higher-dimensional varieties.
\newblock In {\em Algebraic geometry---{S}anta {C}ruz 1995}, volume~62 of {\em
  Proc. Sympos. Pure Math.}, pages 153--183. Amer. Math. Soc., Providence, RI,
  1997.

\bibitem[Bau09]{Bau09}
Thomas Bauer.
\newblock A simple proof for the existence of {Z}ariski decompositions on
  surfaces.
\newblock {\em J. Algebraic Geom.}, 18(4):789--793, 2009.

\bibitem[BC18]{BC18}
Bignalet-Cazalet.
\newblock {\em G{\'e}om{\'e}trie de la projectivisation des id{\'e}aux et
  applications aux probl{\`e}mes de birationalit{\'e}}.
\newblock Thesis, 2018.

\bibitem[BCHM10]{BCHM10}
Caucher Birkar, Paolo Cascini, Christopher~D. Hacon, and James McKernan.
\newblock Existence of minimal models for varieties of log general type.
\newblock {\em J. Amer. Math. Soc.}, 23(2):405--468, 2010.

\bibitem[BDPP13]{BDPP13}
S{\'e}bastien Boucksom, Jean-Pierre Demailly, Mihai P{\u a}un, and Thomas
  Peternell.
\newblock The pseudo-effective cone of a compact {K\"a}hler manifold and
  varieties of negative {K}odaira dimension.
\newblock {\em Journal of Algebraic Geometry}, 22:201--248, 2013.

\bibitem[BHPVdV04]{BHPV04}
Wolf~P. Barth, Klaus Hulek, Chris A.~M. Peters, and Antonius Van~de Ven.
\newblock {\em Compact complex surfaces}, volume~4 of {\em Ergebnisse der
  Mathematik und ihrer Grenzgebiete. 3. Folge.}
\newblock Springer-Verlag, Berlin, second edition, 2004.

\bibitem[BKT13]{BKT13}
Yohan Brunebarbe, Bruno Klingler, and Burt Totaro.
\newblock Symmetric differentials and the fundamental group.
\newblock {\em Duke Math. J.}, 162(14):2797--2813, 2013.

\bibitem[Bou04]{Bou04}
S{\'e}bastien Boucksom.
\newblock Divisorial {Z}ariski decompositions on compact complex manifolds.
\newblock {\em Ann. Sci. \'Ecole Norm. Sup. (4)}, 37(1):45--76, 2004.

\bibitem[Bro14]{Bro14}
Damian Brotbek.
\newblock Hyperbolicity related problems for complete intersection varieties.
\newblock {\em Compos. Math.}, 150(3):369--395, 2014.

\bibitem[Bru06]{Bru06}
Marco Brunella.
\newblock A positivity property for foliations on compact {K}\"{a}hler
  manifolds.
\newblock {\em Internat. J. Math.}, 17(1):35--43, 2006.

\bibitem[BTVA19]{BTV19}
Nils Bruin, Jordan Thomas, and Anthony Varilly-Alvarado.
\newblock Explicit computation of symmetric differentials and its application
  to quasi-hyperbolicity.
\newblock {\em arXiv:1912.08908}, 2019.

\bibitem[Cam98]{Cam98}
Fr{\'e}d{\'e}ric Campana.
\newblock Negativity of compact curves in infinite covers of projective
  surfaces.
\newblock {\em J. Algebraic Geom.}, 7(4):673--693, 1998.

\bibitem[Cam04]{Cam04}
Fr{\'e}d{\'e}ric Campana.
\newblock Orbifolds, special varieties and classification theory.
\newblock {\em Ann. Inst. Fourier (Grenoble)}, 54(3):499--630, 2004.

\bibitem[Cam11]{Cam11}
Fr\'{e}d\'{e}ric Campana.
\newblock Orbifoldes g\'{e}om\'{e}triques sp\'{e}ciales et classification
  bim\'{e}romorphe des vari\'{e}t\'{e}s k\"{a}hl\'{e}riennes compactes.
\newblock {\em J. Inst. Math. Jussieu}, 10(4):809--934, 2011.

\bibitem[CDP15]{CDP12}
F.~Campana, J.-P. Demailly, and Th. Peternell.
\newblock Rationally connected manifolds and semipositivity of the {R}icci
  curvature.
\newblock In {\em Recent advances in algebraic geometry}, volume 417 of {\em
  London Math. Soc. Lecture Note Ser.}, pages 71--91. Cambridge Univ. Press,
  Cambridge, 2015.

\bibitem[CZ79]{CZ79}
David~A. Cox and Steven Zucker.
\newblock Intersection numbers of sections of elliptic surfaces.
\newblock {\em Invent. Math.}, 53(1):1--44, 1979.

\bibitem[Deb01]{Deb01}
Olivier Debarre.
\newblock {\em Higher-dimensional algebraic geometry}.
\newblock Universitext. Springer-Verlag, New York, 2001.

\bibitem[DPS94]{DPS94}
Jean-Pierre Demailly, Thomas Peternell, and Michael Schneider.
\newblock Compact complex manifolds with numerically effective tangent bundles.
\newblock {\em J. Algebraic Geom.}, 3(2):295--345, 1994.

\bibitem[Dru18]{Dru18}
St\'{e}phane Druel.
\newblock A decomposition theorem for singular spaces with trivial canonical
  class of dimension at most five.
\newblock {\em Invent. Math.}, 211(1):245--296, 2018.
\newblock doi: 10.1007/s00222-017-0748-y.

\bibitem[FK80]{FK80}
Hershel~M. Farkas and Irwin Kra.
\newblock {\em Riemann surfaces}, volume~71 of {\em Graduate Texts in
  Mathematics}.
\newblock Springer-Verlag, New York-Berlin, 1980.

\bibitem[GHS03]{GHS03}
Tom Graber, Joe Harris, and Jason Starr.
\newblock Families of rationally connected varieties.
\newblock {\em J. Amer. Math. Soc.}, 16(1):57--67 (electronic), 2003.

\bibitem[GKKP11]{GKKP11}
Daniel Greb, Stefan Kebekus, S\'andor~J. Kov\'acs, and Thomas Peternell.
\newblock Differential forms on log canonical spaces.
\newblock {\em Publ. Math. Inst. Hautes \'Etudes Sci.}, (114):87--169, 2011.

\bibitem[GKP16]{GKP16b}
Daniel Greb, Stefan Kebekus, and Thomas Peternell.
\newblock Etale fundamental groups of {K}awamata log terminal spaces, flat
  sheaves, and quotients of abelian varieties.
\newblock {\em Duke Math. J.}, 165(10):1965--2004, 2016.

\bibitem[Har77]{Har77}
Robin Hartshorne.
\newblock {\em Algebraic geometry}.
\newblock Springer-Verlag, New York, 1977.
\newblock Graduate Texts in Mathematics, No. 52.

\bibitem[HLS20]{HLS20}
Andreas H\"{o}ring, Jie Liu, and Feng Shao.
\newblock Examples of {F}ano manifolds with non-pseudoeffective tangent bundle.
\newblock {\em arXiv preprint}, 2003.09476, 2020.

\bibitem[HP19]{HP19}
Andreas H\"{o}ring and Thomas Peternell.
\newblock Algebraic integrability of foliations with numerically trivial
  canonical bundle.
\newblock {\em Invent. Math.}, 216(2):395--419, 2019.

\bibitem[KM98]{KM98}
J{\'a}nos Koll{\'a}r and Shigefumi Mori.
\newblock {\em Birational geometry of algebraic varieties}, volume 134 of {\em
  Cambridge Tracts in Mathematics}.
\newblock Cambridge University Press, Cambridge, 1998.
\newblock With the collaboration of C. H. Clemens and A. Corti.

\bibitem[Kol96]{Ko96}
J{\'a}nos Koll{\'a}r.
\newblock {\em Rational curves on algebraic varieties}, volume~32 of {\em
  Ergebnisse der Mathematik und ihrer Grenzgebiete. 3. Folge. A Series of
  Modern Surveys in Mathematics}.
\newblock Springer-Verlag, Berlin, 1996.

\bibitem[Laz04]{Laz04a}
Robert Lazarsfeld.
\newblock {\em Positivity in algebraic geometry. {I}}, volume~48 of {\em
  Ergebnisse der Mathematik und ihrer Grenzgebiete.}
\newblock Springer-Verlag, Berlin, 2004.
\newblock Classical setting: line bundles and linear series.

\bibitem[Mic64]{Mic64}
Artibano Micali.
\newblock Sur les alg\`ebres universelles.
\newblock {\em Ann. Inst. Fourier (Grenoble)}, 14(fasc., fasc. 2):33--87, 1964.

\bibitem[Nak04]{Nak04}
Noboru Nakayama.
\newblock {\em Zariski-decomposition and abundance}, volume~14 of {\em MSJ
  Memoirs}.
\newblock Mathematical Society of Japan, Tokyo, 2004.

\bibitem[Sak79]{Sak79}
Fumio Sakai.
\newblock Symmetric powers of the cotangent bundle and classification of
  algebraic varieties.
\newblock In {\em Algebraic geometry ({P}roc. {S}ummer {M}eeting, {U}niv.
  {C}openhagen, {C}openhagen, 1978)}, volume 732 of {\em Lecture Notes in
  Math.}, pages 545--563. Springer, Berlin, 1979.

\bibitem[Ser96]{Ser96}
Fernando Serrano.
\newblock Isotrivial fibred surfaces.
\newblock {\em Ann. Mat. Pura Appl. (4)}, 171:63--81, 1996.

\bibitem[Uen75]{Ue75}
Kenji Ueno.
\newblock {\em Classification theory of algebraic varieties and compact complex
  spaces}.
\newblock Springer-Verlag, Berlin, 1975.
\newblock Notes written in collaboration with P. Cherenack, Lecture Notes in
  Mathematics, Vol. 439.

\end{thebibliography}

\end{document}